\documentclass{amsart}
\usepackage{mathptmx}
\usepackage{graphicx}
\usepackage[english]{babel}
\usepackage{stmaryrd}
\usepackage{amsmath}
\usepackage{amsfonts}
\usepackage{amssymb}
\usepackage{url}
\usepackage{tikz}
\usetikzlibrary{3d}

\newcommand{\R}{\mathbb{R}}
\newcommand{\Ha}{\mathcal{H}}
\newcommand{\finmasscur}{\mathbf{M}}

\DeclareMathOperator{\Tan}{Tan}
\DeclareMathOperator{\ap}{ap}
\DeclareMathOperator{\dil}{dil}
\DeclareMathOperator{\dist}{dist}
\DeclareMathOperator{\dia}{diam}
\DeclareMathOperator{\Mass}{\mathbf{M}}

\DeclareMathOperator{\ald}{\ap\dil}

\DeclareMathOperator{\set}{set}

\DeclareMathOperator{\supp}{supp}
\DeclareMathOperator*{\esssup}{ess\,sup}

\newcommand{\N}{{\mathbb N}}

\newcommand{\Lip}{\mathrm{Lip}}

\newcommand{\lebmeas}{{\mathcal{L}}}
\numberwithin{equation}{section}

\newtheorem{theorem}{Theorem}[section]
\newtheorem{lemma}[theorem]{Lemma}
\newtheorem{proposition}[theorem]{Proposition}
\newtheorem{definition}[theorem]{Definition}
\newtheorem{corollary}[theorem]{Corollary}
\newtheorem{remark}[theorem]{Remark}

\title{Semicontinuity of eigenvalues under intrinsic flat convergence}

\begin{document}

\author{Jacobus W. Portegies}
\address{Max Planck Institute for Mathematics in the Sciences, Leipzig, Germany}
\email{jacobus.portegies@mis.mpg.de}

\begin{abstract}
We use the theory of rectifiable metric spaces to define a Dirichlet energy of Lipschitz functions defined on the support of integral currents. This energy is obtained by integration of the square of the norm of the tangential derivative, or equivalently of the approximate local dilatation, of the Lipschitz functions. We define min-max values based on the normalized energy and show that when integral current spaces converge in the intrinsic flat sense without loss of volume, the min-max values of the limit space are larger than or equal to the upper limit of the min-max values of the currents in the sequence. In particular, the infimum of the normalized energy is semicontinuous. On spaces that are infinitesimally Hilbertian, we can define a linear Laplace operator. We can show that semicontinuity under intrinsic flat convergence holds for eigenvalues below the essential spectrum, if the total volume of the spaces converges as well.
%\subclass{49Q15 \and 58J50 \and 49Q20 \and 53C23}
\end{abstract}

\maketitle

\section{Introduction}

Riemannian manifolds come with a natural definition of a Laplace operator, which is involved in describing heat flow or diffusion on the manifold.
The properties of the Laplace operator, such as its eigenvalues and eigenfunctions, give some information on the geometry of the manifold, and vice versa (see for instance \cite{berard_spectral_1986}, \cite{berger_spectre_1971} and \cite{barnett_spectral_2012}). In this context it is natural to ask whether the spectra of the Laplace operator on the manifolds converge when the manifolds do.

It is easy to see that the spectrum of the Laplace operator varies continuously under $C^2$-convergence of manifolds. As for weaker types of convergence, Fukaya has constructed examples demonstrating that under Gromov-Hausdorff convergence, the eigenvalues of the Laplace operator need not be continuous \cite{fukaya_collapsing_1987}. 
One way to retrieve continuity is to keep track of a measure on the space as well, that is, to consider Fukaya's metric measure convergence rather than just Gromov-Hausdorff convergence. Fukaya showed that with uniform bounds on the sectional curvature and the diameter, the eigenvalues of the Laplace operator on manifolds are continuous under measured Gromov-Hausdorff convergence \cite{fukaya_collapsing_1987}. 
He also conjectured that the assumption of uniform bounds on the sectional curvature could be replaced by a lower bound on the Ricci curvature. He proved that without curvature bounds, the eigenvalues are still upper semicontinuous \cite{fukaya_collapsing_1987}.

Cheeger and Colding proved Fukaya's conjecture, in that they showed that with a uniform lower bound on the Ricci curvature, a Laplace operator can be defined on the limit space, and the eigenvalues and eigenfunctions are continuous \cite{cheeger_structure_2000}. Moreover, these limit spaces are almost everywhere Euclidean.

Rather than considering metric measure convergence, we will study what happens to the eigenvalues of the Laplace operator when the underlying manifolds converge in the intrinsic flat sense. 

Ambrosio and Kirchheim showed how one can define currents on metric spaces \cite{ambrosio_currents_2000}, following an idea by De Giorgi \cite{de_giorgi_general_1995}. Sormani and Wenger \cite{sormani_intrinsic_2011} applied their work to define an intrinsic flat distance between oriented Riemannian manifolds of finite volume (and more generally, between integral current spaces). The flat distance was originally introduced by Whitney for submanifolds of Euclidean space, and extended to integral currents by Federer and Fleming \cite{federer_normal_1960}. Similar to how the Gromov-Hausdorff distance relates to the Hausdorff distance, the intrinsic flat distance between two manifolds is determined as the infimum of the flat distance between isometric embeddings of the manifolds taken over all isometric embeddings into all possible common metric spaces.

There are important differences between metric-measure convergence and intrinsic flat convergence. The metric measure limit of compact metric measure spaces is always compact, while this need not be the case under intrinsic flat convergence. Limits obtained under intrinsic flat convergence are always rectifiable.
Moreover, there are examples of sequences of spaces that do not have a Gromov-Hausdorff limit, but do have an intrinsic flat limit \cite{sormani_intrinsic_2011}. 

Sequences of manifolds that converge in the intrinsic flat sense occur naturally, since by Wenger's compactness theorem for integral currents on metric spaces \cite{wenger_compactness_2011}, a sequence of integral current spaces with bounded diameter, mass and mass of the boundary has a subsequence converging in the intrinsic flat sense. In other words, a sequence of oriented Riemannian manifolds with boundary, with uniform bounds on the volumes and the volumes of the boundary, will have a subsequence converging in the intrinsic flat sense to an integral current space.

The following theorem presents our first main result formulated for closed, oriented Riemannian manifolds, instead of more general integral currents.

\begin{theorem}
\label{th:main}
Let $M_i$ ($i=1,2,\dots$) and $M$ be closed, oriented Riemannian manifolds such that as $i \to \infty$, $M_i$ converges to $M$ in the intrinsic flat sense.
Moreover, assume that $\mathrm{Vol}(M_i) \to \mathrm{Vol}(M)$. 
Then, for $k=1,2,\dots$,
\begin{equation}
\label{eq:semicont}
\limsup_{i \to \infty} \lambda_k(M_i) \leq \lambda_k(M),
\end{equation}
where $\lambda_k(\tilde{M})$ is the $k$th eigenvalue of the Laplace operator on a manifold $\tilde{M}$.
\end{theorem}

The precise definition of min-max values $\lambda_k(M)$ when $M$ is \emph{not} a Riemannian manifold, but an integral current, is given through a min-max variational problem involving a Dirichlet energy in (\ref{eq:minmax}). We show semicontinuity of $\lambda_k$ in Theorem \ref{th:SemContMinMaxFlat} under flat convergence without loss of volume and in Theorem \ref{Th:SemContFirst} we prove semicontinuity of $\lambda_k$ under intrinsic flat convergence without loss of volume. In particular, it follows that the infimum of the normalized energy is semicontinuous.

It follows from the work by Kirchheim \cite{kirchheim_rectifiable_1994} and Ambrosio and Kirchheim \cite{ambrosio_rectifiable_2000} that rectifiable currents $T$ with associated mass measures $\|T\|$ are concentrated on a set that in a measure-theoretic sense is locally Finsler, and Lipschitz functions $f$ on these sets are almost-everywhere tangentially differentiable, with tangential derivative $df$. 
Even for Finsler spaces, there does not seem to be a canonical choice of a Laplace operator, see for instance the works by Bao and Lackey \cite{bao_hodge_1996}, Shen \cite{shen_non-linear_1998}, Centore \cite{centore_mean-value_1998} and Barthelm\'{e} \cite{barthelme_natural_????}.
In Section \ref{se:NormEn}, we will use the local Finsler structure and the tangential derivatives of Lipschitz functions to define a Dirichlet energy of Lipschitz functions $f$ defined on an integral current $T$ by
\begin{equation}
E_T(f) = \int_X |df|^2 d\|T\|.
\end{equation}
We will also give an alternative expression in terms of an approximate local dilatation.
For Riemannian manifolds, this energy corresponds to the usual Dirichlet energy. 
A similar definition of the energy for instance appears in \cite{gromov_dimension_1988}. This definition also corresponds to the energy that is at the basis of the non-linear Laplacian as introduced by Shen \cite{shen_non-linear_1998}, when restricted to Finsler manifolds, except that our definition differs from the one used by Gromov \cite{gromov_dimension_1988} and Shen \cite{shen_non-linear_1998} in in that we will use the mass measure as introduced by Ambrosio and Kirchheim \cite{ambrosio_currents_2000}, which is the natural measure on an integral current space, rather than the Hausdorff measure. 
The energy will also turn out to equal the Cheeger energy \cite{cheeger_differentiability_1999} for the metric measure space $(X, d_X, \|T\|)$, as we will see from comparing with the definition by Ambrosio, Gigli and Savar\'{e} \cite{ambrosio_calculus_2014}.

After defining the energy $E_T$, we consider the completion of bounded Lipschitz functions under the Sobolev norm $\|.\|_{W^{1,2}(\|T\|)} = \|. \|_{L^2(\|T\|)} + E_T(.)$. 
In Theorem \ref{th:TangDiffSob} we show that if $T$ is an integral current on a $w^*$-separable dual Banach space $Y$, the objects in this completion are also $\|T\|$-a.e. tangentially differentiable, with a natural definition of the tangential derivative. 
We will use this to conclude that for an integral current $T$ on a complete metric space $X$, $W^{1,2}(\|T\|)$ can be interpreted in a natural way as a subset of $L^2(\|T\|)$. 
Moreover, we will show in Theorem \ref{th:LowSCEnergy} that the energy $E_T$ is lower semicontinuous. 
In fact, we show in Theorem \ref{th:ApdilIsMinRelaxGrad} that $f \in W^{1,2}(\|T\|)$ if and only if it has a minimal relaxed gradient in the sense of Ambrosio, Gigli and Savar\'{e} \cite{ambrosio_calculus_2014}, and that the norm of the tangential derivative is equal to the minimal relaxed gradient. 

An important ingredient in the proof is a Poincar\'{e}-like inequality, that is quite technical and is shown in Appendix \ref{se:Poincare}. 
As a by-product, we also record a decomposition theorem for one-dimensional integral currents in Appendix \ref{se:Decomposition}, that is a simpler version of a recent decomposition theorem by Paolini and Stepanov for one-dimensional normal currents \cite{paolini_decomposition_2012,paolini_structure_2013}.

When the integral currents involved are supported on an infinitesimally Hilbertian rectifiable metric space, we can define an unbounded (linear) self-adjoint operator on their rectifiable sets. 
The min-max values of these operators correspond to the $\lambda_k$ defined before, and consequently, all $\lambda_k$ below the essential spectrum correspond to eigenvalues of the operators.

The structure of this paper is as follows. In Section \ref{se:background}, we review preliminaries from the work by Kirchheim \cite{kirchheim_rectifiable_1994}, Ambrosio and Kirchheim \cite{ambrosio_currents_2000,ambrosio_rectifiable_2000}, and Sormani and Wenger \cite{sormani_intrinsic_2011}. In Section \ref{se:NormEn}, we introduce the energy and give two possible expressions for it. 
In Section \ref{se:DefSobolev}, we introduce the Sobolev space $W^{1,2}(\|T\|)$ for integral currents $T$, and we will show that every element in the space is ($\|T\|$-a.e.) tangentially differentiable. 
It follows that the space can be understood as a subset of $L^2(\|T\|)$.
In Section \ref{se:LowSemEnergy}, we show the lower semicontinuity of the energy, and we show that the approximate local dilatation corresponds to the minimal relaxed gradient as introduced by Ambrosio, Gigli and Savar\'{e} \cite{ambrosio_calculus_2014}.
Subsequently, we introduce the min-max functionals $\lambda_k$ in Section \ref{se:minmax}, and show their semicontinuity. 
We then show semicontinuity of the $\lambda_k$ under intrinsic flat convergence in Section \ref{se:SemInfEn}, and the semicontinuity of eigenvalues for infinitesimally Hilbertian spaces in Section \ref{se:InfHilb}. Finally, in Section \ref{se:example}, we show by example that we cannot remove the condition on convergence of the volumes. 
In the appendix, we show the decomposition theorem for one-dimensional integral currents, and derive a Poincar\'{e}-like inequality.

\section*{Acknowledgments} 

I would like to thank Carolyn Gordon for suggesting that one should investigate the behavior of eigenvalues under intrinsic flat convergence and thank Kenji Fukaya for suggesting to prove semicontinuity of the eigenvalues. I am grateful to Christina Sormani, who recommended the problem to me and provided feedback on a draft version of the manuscript. 
I would like to thank Larry Guth and Or Hershkovits for helpful discussions. 
I would like to thank my advisor, Fanghua Lin, for helpful conversations and for teaching me geometric measure theory. 
The main part of the manuscript was written at the Courant Institute, but it was finished at the Max Planck Institute for Mathematics in the Sciences. I would like to thank both institutes for their hospitaility.
I am indebted to the referee, as he / she had very helpful comments.

\section{Background}
\label{se:background}

In this section, we will summarize the properties of currents on metric spaces that we will use. Currents on metric spaces were introduced by Ambrosio and Kirchheim \cite{ambrosio_currents_2000}, following an idea by De Giorgi \cite{de_giorgi_general_1995}. 

\subsection{Currents on metric spaces}

We first review the main concepts from the paper by Ambrosio and Kirchheim on currents on metric spaces \cite{ambrosio_currents_2000}.

Let $Z$ be a complete metric space. For a positive integer $n$, let $\mathcal{D}^n(Z)$ denote the set of all $(n+1)$-tuples $\omega = (f , \pi_1, \dots, \pi_n)$, such that $f$ is bounded and Lipschitz and the $\pi_i$ are Lipschitz. 
In a smooth setting, $\omega$ would correspond to $f d\pi_1 \wedge \dots \wedge d\pi_n$. 

The exterior differential $d$ is an action that creates an $(n+2)$-tuple out of an $(n+1)$-tuple as follows
\begin{equation}
d( f, \pi_1, \dots, \pi_n ) 
= (1, f, \pi_1, \dots, \pi_n).
\end{equation}
If $\phi$ is a Lipschitz map from $Z$ to another metric space $E$, we define the pullback $\phi^\# \omega \in \mathcal{D}^n(Z)$ of $\omega \in \mathcal{D}^n(E)$ by
\begin{equation}
\phi^\# \omega 
= \phi^\# (f, \pi_1, \dots, \pi_n) 
= (f \circ \phi, \pi_1 \circ \phi, \dots, \pi_n \circ \phi).
\end{equation}

An $n$-dimensional metric functional is a function $T: \mathcal{D}^n(Z) \to \mathbb{R}$ that is subadditive and positively $1$-homogeneous with respect to $f$ and the $\pi_i$, $i=1, \dots, n$.

The boundary $\partial T$ of a metric functional $T$ is defined by
\begin{equation}
\partial T (\omega) = T ( d \omega ), \qquad \omega \in \mathcal{D}^n(Z).
\end{equation}
and the push-forward of $T$ under a map $\phi$ is given by
\begin{equation}
\phi_\# T (\omega) = T( \phi^\# \omega ), \qquad \omega \in \mathcal{D}^n(Z).
\end{equation}
A metric functional $T$ is said to have finite mass if there exists a finite Borel measure $\mu$ such that for every $(f, \pi_1, \dots, \pi_n) \in \mathcal{D}^n(Z)$,
\begin{equation}
\label{eq:massbound}
|T(f, \pi_1, \dots, \pi_n) | \leq \prod_{i=1}^n \mathrm{Lip}(\pi_i) \int_Z |f| d \mu.
\end{equation}
The mass of $T$, which we denote by $\|T\|$, is defined as the minimal $\mu$ satisfying (\ref{eq:massbound}).

\begin{definition}[\cite{ambrosio_currents_2000}]
An $n$-current on a metric space $Z$ is defined to be a metric functional $T$ with the additional properties that
\begin{enumerate}
\item $T$ is multilinear in $(f, \pi_1, \dots, \pi_n)$,
\item $\lim_{i \to \infty} T(f, \pi_1^i, \dots, \pi_n^i) = T (f, \pi_1, \dots, \pi_n)$ whenever $\pi_j^i \to \pi_j$ pointwise in $Z$ with $\mathrm{Lip}(\pi_j^i) \leq C$ for some constant $C$.
\item $T(f, \pi_1, \dots, \pi_n) = 0$ if for some $i \in \{ 1, \dots, n \}$ the function $\pi_i$ is constant on a neighborhood of $\{ f \neq 0 \}$.
\end{enumerate}

We denote by $\mathbf{M}_n(Z)$ the Banach space of $n$-dimensional currents on $Z$ with finite mass, with norm $\mathbf{M}(T) = \|T\|(Z)$.
\end{definition}

A function $g \in L^1(\mathbb{R}^n)$ induces a current $\llbracket g \rrbracket \in \mathbf{M}_n(\mathbb{R}^n)$ by
\begin{equation}
\llbracket g \rrbracket (f, \pi_1 , \dots ,\pi_n) 
:= \int_{\mathbb{R}^n} g f \det( \nabla \pi ) \, dx.
\end{equation}

A normal current is a current $T \in \mathbf{M}_n(Z)$ such that $\partial T \in \mathbf{M}_{n-1}(Z)$. 
A current $T$ is called rectifiable if $\|T\|$ is concentrated on a countably $\Ha^n$-rectifiable set, and vanishes on $\Ha^n$-negligible Borel sets. 
We define 
\begin{equation}
\set(T) := \left\{ x \in Z \, | \, \liminf_{r \downarrow 0} \frac{\|T\|(B(x,r))}{r^n} > 0 \right\}.
\end{equation}
For rectifiable currents, $\set(T)$ is a rectifiable set.
It is said to be \emph{integer} rectifiable if, in addition, for any $\phi \in \Lip(Z, \R^n)$ and any open $O \subset Z$, $\phi_\# (T \llcorner O) = \llbracket g \rrbracket$ for some $g \in L^1 (\R^n, \mathbb{Z})$. 
\emph{Integral} currents are integer rectifiable currents that are also normal. 
We denote the class of $n$-dimensional integral currents on $Z$ by $\mathbf{I}_n(Z)$.
The boundary operator $\partial$ maps from $I_{n+1}(Z)$ to $I_n(Z)$.

We say that a sequence of $n$-currents $T_1, T_2, \dots$ converges weakly to an $n$-current $T$ if for all $\omega \in \mathcal{D}^n(Z)$, $T_i(\omega) \to T(\omega)$.
The mass is lower semicontinuous under weak convergence \cite{ambrosio_currents_2000}.
That is, if $T_i \rightharpoonup T$ weakly, then for every $O$ open,
\begin{equation}
\liminf_{i \to \infty} \|T_i\|(O) \geq \|T\|(O).
\end{equation}

\subsection{The flat distance}

Federer and Fleming \cite{federer_normal_1960} extended the concept of flat distance to Euclidean integral currents. 
In \cite{wenger_flat_2007}, Wenger explores the properties of the following analogous flat distance between integral currents on metric spaces.

\begin{definition}[\cite{wenger_flat_2007}]
The flat distance between $T_1 \in \mathbf{I}_n(Z)$ and $T_2 \in \mathbf{I}_n(Z)$ is given by
\begin{equation}
d_{F}^Z (T_1, T_2 ) = 
\inf\left\{ \mathbf{M}(U) + \mathbf{M}(V) \,| \, T_1 - T_2 = U + \partial V \right\}
\end{equation}
where the $\inf$ is taken over all $U \in \mathbf{I}_n(Z)$ and $V \in \mathbf{I}_{n+1}(Z)$.
\end{definition}

\subsection{The intrinsic flat distance}

Subsequently, Sormani and Wenger introduced the concept of integral current spaces, and an \emph{intrinsic flat distance} between them.

\begin{definition}[\cite{sormani_intrinsic_2011}]
An $n$-dimensional integral current space $(X,d,T)$ is a triple of a separable metric space $X$ with a distance $d$ and an integer rectifiable current $T\in \mathbf{I}_n(\overline{X})$ on the completion $\overline{X}$ of $X$, with the additional condition that 
\begin{equation}
X = \left\{ x \in \overline{X} \, | \, \liminf_{r \downarrow 0} \frac{\|T\|(B(x,r))}{r^n} > 0 \right\}.
\end{equation}
\end{definition}

An $n$-dimensional oriented manifold $M$ induces an integral current space denoted by $\llbracket M \rrbracket=(M,d_M,T)$, with $d_M$ the geodesic distance on $M$, by 
\begin{equation}
T(f,\pi_1, \dots, \pi_n) = \int f  \, 
d\pi_1 \wedge \dots \wedge d\pi_n.
\end{equation}

Imitating Gromov's definition of the Gromov-Hausdorff distance, Sormani and Wenger \cite{sormani_intrinsic_2011} introduced the intrinsic flat distance as follows.

\begin{definition}[\cite{sormani_intrinsic_2011}]
The intrinsic flat distance between two integral current spaces $M_i = (X_i, d_i, T_i)$, $i=1,2$, is defined by
\begin{equation}
d_\mathcal{F}( M_1, M_2 ) := 
\inf d_F^Z \left( (\phi_1)_\# T_1 , (\phi_2)_\# T_2 \right)
\end{equation}
where the infimum is over all complete metric spaces $Z$ and all \emph{isometric} embeddings $\phi_i : X_i \to Z$.
\end{definition}

As in Gromov's definition, by an \emph{isometric} embedding from a metric space $X$ with distance $d_X$ into a metric space $Z$ with distance $d_Z$ we mean a map $I:X \to Z$ such that for every $x, y \in X$,
\begin{equation}
d_X(x,y) = d_Z(I(x),I(y)).
\end{equation}

We will also use the following result by Sormani and Wenger \cite{sormani_intrinsic_2011}.

\begin{theorem}[\cite{sormani_intrinsic_2011}]
\label{th:embcommonspace}
Let $M_i = (X_i, d_i, T_i)$ be a sequence of integral current spaces converging in the intrinsic flat sense to a limit integral current space $M = (X,d,T)$. 
Then there exist a complete, separable metric space $Z$, and isometric embeddings $\phi_i: \overline{X_i} \to Z$, $\phi: \overline{X} \to Z$ such that $(\phi_i)_\# T_i \to \phi_\# T$ in the flat distance in $Z$.
\end{theorem}

\begin{definition}
Let $X$ be a metric space and let $\Sigma$ denote its Borel $\sigma$-algebra.
We say that a sequence $\mu_n$ of finite measures on $(X, \Sigma)$ converges \emph{weakly} to a finite measure $\mu$ on $(X,\Sigma)$, and write $\mu_n \rightharpoonup \mu$ weakly, if for all bounded and continuous functions $f : X \to \R$,
\begin{equation}
\int_X f d \mu_n \to \int_X f d \mu.
\end{equation}
\end{definition}

We finally include the following simple lemma for later use.

\begin{lemma}
\label{le:weakcurweakmeas}
Suppose $Z$ is a complete metric space, and $T_i \in \finmasscur_n(Z)$, $(i=1,2,\dots$) and $T\in \finmasscur_n(Z)$ such that $T_i \rightharpoonup T$ weakly. 
Moreover, assume that $\Mass(T_i) \to \Mass(T)$. 
Then $\|T_i\| \rightharpoonup \|T\|$ weakly as measures.
\end{lemma}

\begin{proof}
Since $T_i \rightharpoonup T$ weakly, the mass is lower semicontinuous, in that 
\begin{equation}
\liminf_{i \to \infty} \|T_i\|(O) \geq \|T\|(O),
\end{equation}
for any $O \subset Z$ open \cite{ambrosio_currents_2000}. 
By assumption, $\|T_i\|(Z) \to \|T\|(Z)$. 
Since $Z$ is a complete metric space, it follows that $\|T_i\| \to \|T\|$ by for instance the portmanteau theorem (cf. \cite{klenke_probability_2007}).
\end{proof}

\subsection{Rectifiable sets in metric and Banach spaces}

In this section, we will review some important results on rectifiable sets on metric spaces, that were obtained mainly by Kirchheim \cite{kirchheim_rectifiable_1994} and Ambrosio and Kirchheim \cite{ambrosio_rectifiable_2000}. 
First of all, we review two concepts of differentiability, namely metric differentiability and $w^*$-differentiability.

\begin{definition}[{\cite[Definition 3.1]{ambrosio_rectifiable_2000}}]
Let $Z$ be a metric space. 
A function $g:\R^n \to Z$ is called metrically differentiable at a point $x \in \R^n$ if there is a seminorm $md_xg(.)$ on $\R^n$ such that
\begin{equation}
d(g(y), g(x)) - md_x g (y - x) = o(|y - x|), 
\end{equation}
as $y \to x$. 
We call $md_x g$ the metric differential of $g$ at $x$.
\end{definition}

\begin{theorem}[{\cite[Theorem 3.2]{ambrosio_rectifiable_2000}}]
Any function $g: \R^n \to Z$, with $Z$ a metric space, is metrically differentiable at $\lebmeas^n$-a.e. $x \in \R^n$.
\end{theorem}

Before we recall the $w^*$-differentiability, let us first specify what we mean by a $w^*$-separable Banach space.

\begin{definition}
By a $w^*$-separable Banach space $Y$ we will denote a dual Banach space $Y = G^*$, for a separable Banach space $G$.
\end{definition}

An important example of a $w^*$-separable Banach space is the space $\ell^\infty$. 
Any separable metric space $X$ can be isometrically embedded into $\ell^\infty$ by the Kuratowski embedding.
If $x_i$ is a dense sequence in $X$, such an embedding is given by
\begin{equation}
(I(x))_j = d(x,x_j) - d(x_0, x_j).
\end{equation}

\begin{definition}[{\cite[Definition 3.4]{ambrosio_rectifiable_2000}}]
Let $Y$ be a $w^*$-separable dual space, and let $g: \R^n \to Y$. 
We say that $g$ is $w^*$-differentiable at $x\in\R^n$ if there is a linear map $wd_xg:\R^n \to Y$ such that
\begin{equation}
w^*-\lim_{y\to x} \frac{g(y) - g(x) - wd_xg (y-x)}{|y-x|} = 0.
\end{equation}
The map $wd_x g$ is called the $w^*$-differential of $g$ at $x$.
\end{definition}

\begin{theorem}[{\cite[Theorem 3.5]{ambrosio_rectifiable_2000}}]
Let $Y$ be a $w^*$-separable Banach space.
Any Lipschitz function $g: \R^n \to Y$ is metrically and $w^*$-differentiable and fulfils
\begin{equation}
md_xg(v) = \|wd_xg(v)\|, \qquad \text{ for all } v \in \R^n,
\end{equation}
for $\lebmeas^n$-a.e. $x \in \R^n$.
\end{theorem}

Next, we recall the definition of the Jacobian of a linear map between two Banach spaces.

\begin{definition}[{\cite[Definition 4.1]{ambrosio_rectifiable_2000}}]
Let $V$ and $W$ be Banach spaces, with $\dim V = n$. Let $L$ be a linear map $L:V \to W$. 
Define the $n$-Jacobian of $L$ by 
\begin{equation}
J_n(L) := \frac{\omega_n}
			   {\Ha^n(\{x \in V \, | \, \| L(x) \| \leq 1 \})}.
\end{equation}
Similarly, when $s$ is a seminorm on $\R^n$, define 
\begin{equation}
J_n(s) := \frac{\omega_n}
	 		   {\Ha^n(\{x \in \R^n \, | \, s(x) \leq 1 \})}.
\end{equation}
Here, $\omega_n$ is the volume of the Euclidean unit ball in $n$ dimensions.
\end{definition}

\begin{definition}
The upper and lower $n$-dimensional densities of a finite Borel measure $\mu$ at a point $x$ are defined respectively as
\begin{equation}
\Theta^*_n (\mu,x):= \limsup_{ r \downarrow 0} \frac{\mu(B_r(x)) }{\omega_n r^n}, \qquad 
\Theta_{*n}(\mu,x):= \liminf_{r \downarrow 0} \frac{\mu(B_r(x))}{\omega_n r^n}.
\end{equation}
When $\Theta^*_n(\mu,x) = \Theta_{*n}(\mu,x)$ in a point $x$ we define the $n$-dimensional density of $\mu$ at $x$ by $\Theta_n(\mu,x) = \Theta^*_n(\mu,x)$. 
\end{definition}

\begin{definition}
\label{de:rectifiability}
A subset $S\subset Z$ is called countably $\Ha^n$-rectifiable if there exists a sequence of Lipschitz functions $g_j: A_j \subset \R^n \to Z$ such that 
\begin{equation}
\Ha^n \left( S \backslash \bigcup_j g_j(A_j) \right) = 0.
\end{equation}
We say that a finite Borel measure $\mu$ is $n$-rectifiable if $\mu = \theta \Ha^n \llcorner S$ for a countably $\Ha^n$-rectifiable set $S$ and a Borel function $\theta:S \to (0,\infty)$.
\end{definition}

\begin{definition}[{\cite[Definition 5.5]{ambrosio_rectifiable_2000}}]
If $Y$ is a $w^*$-separable dual space, and $S\subset Y$ is countably $\Ha^n$-rectifiable, with functions $g_j$ as in Definition \ref{de:rectifiability}, the approximate tangent space to $S$ at a point $x$ is defined as 
\begin{equation}
\Tan(S,x) = wd_y g_i(\R^n),
\end{equation} 
when for some $i\in \mathbb{N}$, $y = g_i^{-1}(x)$ and $g_i$ is metrically and $w^*$-differentiable at $y$, with $J_n(wd_y g_i) > 0$. 
It is shown in \cite{ambrosio_rectifiable_2000} that this is a good definition for $\Ha^n$-a.e. $x \in S$.
\end{definition}

In case $X$ is an arbitrary separable metric space, the approximate tangent space can still be defined by using an isometric embedding $j:X \to Y$ of $X$ into a $w^*$-separable dual space $Y$, and setting
\begin{equation}
\Tan(S,x) = \Tan(j(S),j(x)).
\end{equation}
It is shown that this definition does not depend on the choice of $j$ and $Y$, in the sense that $\Tan(S,x)$ is uniquely determined $\mathcal{H}^k$-a.e. up to linear isometries \cite{ambrosio_rectifiable_2000,ambrosio_currents_2000}.

The next theorem by Ambrosio and Kirchheim \cite{ambrosio_rectifiable_2000} shows the existence of tangential derivatives of Lipschitz functions on rectifiable sets in $w^*$-separable dual spaces. 
In the formulation of the theorem, the distance $d_w$ metrizes the $w^*$-topology of a $w^*$-separable dual space $Y=G^*$, that is, for $x,y \in Y$,
\begin{equation}
d_w(x,y) := \sum_{j=0}^\infty 2^{-j} | \langle x-y, g_j \rangle|,
\end{equation}
where $(g_j)_{j=1}^\infty \subset G$ is a countable dense set in the unit ball of $G$.

\begin{theorem}[{\cite[Theorem 8.1]{ambrosio_rectifiable_2000}}]
\label{th:TangDiff}
Let $S\subset Y$ be a countably $\Ha^n$-rectifiable set of a $w^*$-separable dual space $Y$, and let $f$ be a Lipschitz function from $Y$ into another $w^*$-separable dual space $\tilde{Y}$. 
Let $\theta:S \to (0,\infty)$ be integrable with respect to $\Ha^n \llcorner S$ and let $\mu = \theta \Ha^n \llcorner S$ be the corresponding rectifiable measure.

Then for $\Ha^n$-almost every $x\in S$, there exist a $w^*$-continuous and linear map $L: Y \to \tilde{Y}$, and a Borel set $S^x \subset S$ such that $\Theta_n^*( \mu \llcorner S^x, x) = 0$ and 
\begin{equation}
\lim_{y \in S \backslash S^x \to x} \frac{d_w(f(y),f(x) + L(y-x))}{|y - x|} = 0.
\end{equation}
The map $L$ is uniquely determined on $\Tan(S,x)$. We denote its restriction to $\Tan(S,x)$ by
\begin{equation}
d_x^S f: \Tan(S,x) \to \tilde{Y},
\end{equation}
and call it the tangential differential. 
It is characterized by the property that for any Lipschitz map $g: D \subset \R^n \to S$,
\begin{equation}
wd_y(f \circ g) = d_{g(y)}^S f \circ wd_y g, \qquad \text{for } \mathcal{L}^n\text{-a.e. } y \in D.
\end{equation}
\end{theorem}

Finally, we recall the area formula by Ambrosio and Kirchheim \cite{ambrosio_rectifiable_2000}.

\begin{theorem}[{Area formula \cite[Theorem 8.2]{ambrosio_rectifiable_2000}}]
Let $f:Z \to \tilde{Z}$ be a Lipschitz function and let $S \subset Z$ be a countably $\Ha^n$-rectifiable set. 
Then, for any Borel function $\theta: S \to [0,\infty]$,
\begin{equation}
\int_S \theta(x) J_n(d^S f) d\Ha^n(x) = \int_{\tilde{Z}} \sum_{x\in S \cap f^{-1}(y) } \theta(x) d\Ha^n(y).
\end{equation}
Moreover, for any Borel set $A$ and any Borel function $\theta: \tilde{Z} \to [0,\infty]$,
\begin{equation}
\int_A \theta(g(x)) J_n(d^S f) d\Ha^n(x) 
= \int_{\tilde{Z}} \theta(y) \Ha^0\left(A \cap f^{-1}(y) \right) d \Ha^n (y).
\end{equation}
\end{theorem}
In the Theorem above, $J_n(d^S f)$ is calculated after embedding $Z$ and $\tilde{Z}$ in $w^*$-separable metric spaces and calculating the tangential derivative of the appropriate lift of $f$. We would also like to mention that in special cases it may be easier to apply the less general version of the area formula \cite[Theorem 5.2]{ambrosio_rectifiable_2000}.

\subsection{The John Ellipsoid}
\label{se:JohnEllipsoid}

The John Ellipsoid associated to a convex body is the inscribed maximal-volume ellipsoid. 
We will use the results on this ellipsoid to obtain good charts, and to create comparisons of certain Banach spaces to Hilbert spaces.

\begin{theorem}[John's Ellipsoid Theorem \cite{john_extremum_1948}, see also \cite{ball_ellipsoids_1992}]
\label{th:JohnEllipsoid}
For any convex body in $\R^n$, the inscribed maximal-volume ellipsoid exists and is unique.
The Euclidean ball $B_1(0)$ is the ellipsoid of maximal volume contained in the convex body $C \subset \R^n$ if and only if $B_1(0) \subset C$ and, for some $m \geq n$, there are unit vectors $u_i$ on the boundary of $C$ and positive numbers $c_i$ ($i=1, \dots, m$), for which $\sum_i c_i u_i = 0$ and $\sum_i c_i u_i \otimes u_i = I$, the identity on $\R^n$.
\end{theorem}

The theorem has the following consequence for the existence of equivalent norms in finite-dimensional normed spaces.

\begin{corollary}[cf. \cite{ball_ellipsoids_1992}]
\label{co:JohnEllipsoid}
If $\|.\|$ is a norm on a finite-dimensional Banach space $V$, and $\| . \|_J$ is the Hilbert-space norm associated to the maximal-volume ellipsoid inscribed in the unit ball in $V$, then for all $v \in V$,
\begin{equation}
\frac{1}{\sqrt{n}} \| v \|_J \leq \| v \| \leq \|v\|_J.
\end{equation}
\end{corollary}

\section{The Dirichlet energy}
\label{se:NormEn}

In this section we define a Dirichlet energy of functions on sets of integral currents on metric spaces. 
First, we introduce two quantities in Sections \ref{se:NormTangDer} and \ref{se:ApprLocDil}, respectively the norm of the tangential derivative and the approximate local dilatation. 
In Section \ref{se:EqQuant}, we prove that the two coincide. 
Subsequently, we integrate them in Section \ref{se:TheNormEn} to obtain the normalized energy.

\subsection{The norm of the tangential derivative}
\label{se:NormTangDer}

Let $Y$ be a $w^*$-separable dual space and let $\mu = \theta \Ha^n \llcorner S$ be an $n$-rectifiable measure on $Y$. 
Let $f$ be a Lipschitz function on $S$. 
Since $S$ is rectifiable, $f$ is $\mu$-a.e. tangentially differentiable, see \cite[Theorem 8.1]{ambrosio_rectifiable_2000}. 
If it exists, we denote the tangential derivative of $f$ to $S$ at $x$ by $d_x^Sf$. 
Note that $d_x^Sf$ is a linear functional on $\mathrm{Tan}(S,x)$. 
We denote by $|d_x^S f|$ its dual norm. 

If $X$ is an arbitrary separable metric space and $\mu = \theta \Ha^n \llcorner S$ an $n$-rectifiable measure on $X$, we first isometrically embed $X$ into a $w^*$-separable dual space $Y$. 
Consider two such embeddings: $j_1:X \to Y_1$ and $j_2:X \to Y_2$.
By \cite{ambrosio_rectifiable_2000}, for $\mu$-a.e. $x \in X$, the approximate tangent spaces $\Tan(j_1(S),j_1(x))$ and $\Tan(j_2(S),j_2(x))$ are isometric. 
Therefore, one can give a meaning to $\Tan(S,x)$.
The isometry between the approximate tangent spaces induces an isometry between the dual spaces. 
Consequently, the functionals $d_{j_1(x)}^{j_1(S)} (f \circ j_1^{-1})$ and $d_{j_1(x)}^{j_2(S)} (f\circ j_2^{-1})$ are linked through this isometry, and it makes sense to define
\begin{equation}
|d_x^S f| = |d_{j_1(x)}^{j_1(S)} (f\circ j_1^{-1})|,
\end{equation}
for $\mu$-a.e. $x \in X$. 
Sometimes we will just write $df$ for the tangential derivative, and $|df|$ for its dual norm.

\subsection{Approximate local dilatation}
\label{se:ApprLocDil}

We may also give a different definition that does not mention approximate tangent spaces.

\begin{definition}
\label{de:ApDil}
Let $\mu = \theta \Ha^n \llcorner S$ be a rectifiable measure on a metric space $X$. Let $f$ be a Lipschitz function defined on $S$.
Define the set $D_x(t)$ by
\begin{equation}
D_x(t) = \left\{y \in X\backslash \{x\} \, | \, \frac{|f(y) - f(x)|}{d(y,x)} > t  \right\}.
\end{equation}
Then, we define the approximate local dilatation of $f$ at a point $x \in X$ by
\begin{equation}
\ald_x f := \inf \left\{ t > 0 \, | \, \Theta_n^*(\|T\| \llcorner D_x(t) , x) = 0 \right\}.
\end{equation}
\end{definition}
We note that $\ald(f)$ is a bounded, measurable function. 

\subsection{Norm tangential derivative equals approximate local dilatation}
\label{se:EqQuant}

In this section, we will show that for almost every $x$, the norm of the tangential derivative and the approximate local dilatation actually coincide.

First, however, we will present two lemmas that give a nice parametrization of $S$, and an integral current $T$ respectively.

\begin{lemma}[dividing a rectifiable set]
\label{le:dividemeasure}
Let $\mu = \theta \Ha^n \llcorner S$ rectifiable measure on a $w^*$-separable Banach space $Y$. 
Then there exist compact sets $K_i \subset \mathbb{R}^n$ ($i=1,2,\dots$) and Lipschitz maps $g_i:\R^n \to Y$ such that $g_i |_{K_i}$ is bi-Lipschitz, $g_i(K_i) \subset S$, and

\begin{itemize}
\item $g_i(K_i) \cap g_j(K_j) = \emptyset$ for $i \neq j$.
\item The $g_i(K_i)$ cover $\Ha^n$-almost all of $S$, that is,
\begin{equation}
\Ha^n\left( S \backslash \bigcup_{i=1}^\infty g_i(K_i) \right) = 0.
\end{equation}
\item The function $\theta \circ g_i$ is continuous on $K_i$.
\item For every $x \in K_i$, $g_i$ is metrically differentiable in $x$, with $J_k(md_x g_i) > 0$, and $x \mapsto md_x g_i$ is continuous, for some moduli of continuity $\omega_i$ and all $y,x \in K_i$,
\begin{equation}
\label{eq:UnifApprDist}
| d(g_i(y),g_i(x)) - md_x g_i (y-x) | \leq \omega_i(d(g_i(y),g_i(x))) d(g_i(y),g_i(x)).
\end{equation}
and
\begin{subequations}
\begin{alignat}{2}
\frac{1}{ 2 \sqrt{n} }\|y-x\| & < md_xg_i(y-x) & < \| y - x \|, \\
\frac{1}{ 2 \sqrt{n} } \| y - x\| & < d(g_i(y), g_i(x)) & < \|y - x\|.
\end{alignat}
\end{subequations}
\item For every $i$, the map $g_i$ is $w^*$-differentiable in every $x \in K_i$, and the map $x \mapsto wd_x g_i(v)$ is $w^*$-continuous in $K_i$ for every $v \in \R^n$.
\end{itemize}

\end{lemma}

\begin{proof}
Compact sets $K_i \subset \R^n$ and maps $g_i:K_i \to Y$ that satisfy the first two items can be obtained as in \cite[Lemma 4.1]{ambrosio_currents_2000}. Because $Y$ is a Banach space, we may extend $g_i$ to a function $g_i : \R^n \to Y$ \cite{johnson_extensions_1986}. 
By Lusin's theorem, we may replace each $K_i$ by a countable union of compact sets that cover $K_i$ up to a set of $\Ha^n$-measure zero such that on these new sets the restriction of $\theta \circ g_i$ is continuous. Next, the compact sets thus obtained can each be divided in countably many other compact sets satisfying the other bullet points, except technically (\ref{eq:UnifApprDist}) and (\ref{eq:ComparisonToEuclidean}), by the work of Ambrosio and Kirchheim \cite[Theorem 3.3 and Remark 3.6]{ambrosio_rectifiable_2000} and the area formula \cite[Theorem 8.2]{ambrosio_rectifiable_2000}.

Equation (\ref{eq:UnifApprDist}) is very similar to the conclusion of \cite[Theorem 3.3]{ambrosio_rectifiable_2000}, and is proven in a similar way: for fixed $i$, we apply Egorov's theorem to the sequence of functions 
\begin{equation}
h_k(x) := \sup_{\substack{y \in K_i, y \neq x \\ d(g_i(x),g_i(y)) < 1/k}} \left|\frac{md_x g_i(y-x)}{d(g_i(y),g_i(x))} - 1 \right|,
\end{equation}
that converge to zero on $K_i$, to find sets $K_{ij}$ the union of which covers $K_i$ almost everywhere, and such that for moduli of continuity $\omega_{ij}$ and all $y \in K_{i}$ and $x \in K_{ij}$,
\begin{equation}
| d(g_i(y),g_i(x)) - md_x g_i (y-x) | \leq \omega_{ij}(d(g_i(y),g_i(x))) d(g_i(y),g_i(x)).
\end{equation}

That the inequalities in (\ref{eq:ComparisonToEuclidean}) can be guaranteed, after possibly dividing each compact set thus far obtained into at most countably many other compact sets up to a set of $\Ha^n$-measure zero, then follows from John's Ellipsoid Theorem (see also Corollary \ref{co:JohnEllipsoid}) after a possible linear transformation of the sets $K_i \subset \R^n$. 
The lemma follows after reindexing.
\end{proof}

We also give a slight variation to the previous lemma.

\begin{lemma}
\label{le:GoodParametrization}
Let $Y$ be a $w^*$-separable Banach space and let $T \in I_n(Y)$.
Then $T$ has a parametrization
\begin{equation}
T = \sum_{\ell=1}^\infty \vartheta_\ell (g_\ell)_\# \llbracket \chi_{K_\ell} \rrbracket,
\end{equation}
with $\vartheta_\ell \in \N$, compact sets $K_\ell \subset \R^n$ and functions $g_\ell: \R^n \to Y$ which satisfy
\begin{itemize}
\item $g_i(K_i) \cap g_j(K_j) = \emptyset$ for $i \neq j$.
\item The $g_i(K_i) \subset \set(T)$ and they cover $\Ha^n$-almost all of $\set(T)$, that is,
\begin{equation}
\Ha^n\left( S \backslash \bigcup_{i=1}^\infty g_i(K_i) \right) = 0.
\end{equation}
\item For every $x \in K_i$, $g_i$ is metrically differentiable in $x$, with $J_k(md_x g_i) > 0$, and $x \mapsto md_x g_i$ is continuous, for some moduli of continuity $\omega_i$ and all $y,x \in K_i$,
\begin{equation}
| d(g_i(y),g_i(x)) - md_x g_i (y-x) | \leq \omega_i(d(g_i(y),g_i(x))) d(g_i(y),g_i(x)).
\end{equation}
and
\begin{subequations}
\label{eq:ComparisonToEuclidean}
\begin{alignat}{2}
\frac{1}{ 2 \sqrt{n} }\|y-x\| & < md_xg_i(y-x) & < \| y - x \|, \\
\frac{1}{ 2 \sqrt{n} } \| y - x\| & < d(g_i(y), g_i(x)) & < \|y - x\|.
\end{alignat}
\end{subequations}
\item For every $i$, the map $g_i$ is $w^*$-differentiable in every $x \in K_i$, and the map $x \mapsto wd_x g_i(v)$ is $w^*$-continuous in $K_i$ for every $v \in \R^n$.
\end{itemize}
\end{lemma}

\begin{proof}
By \cite[Theorem 4.5]{ambrosio_currents_2000}, there exist a sequence of compact sets $\tilde{K}_i$, bi-Lipschitz maps $\tilde{f}_i:\tilde{K}_i \to Y$ and functions $\tilde{\theta}_i  \in L^1(\R^n, \mathbb{Z})$ such that 
\begin{equation}
T = \sum_{i=1}^\infty (\tilde{f}_i)_\# \llbracket \tilde{\theta}_i \rrbracket, 
\qquad \Mass(T) = \sum_{i=1}^\infty \Mass( (\tilde{f}_i)_\# \llbracket \tilde{\theta}_i \rrbracket ).
\end{equation}
In the proof of \cite[Theorem 4.5]{ambrosio_currents_2000}, the $\tilde{f}_i$ are chosen such that the images $\tilde{f}_i(\tilde{K}_i)$ are pairwise disjoint.
As $Y$ is a Banach space, we may extend the $\tilde{f}_i$ to Lipschitz functions defined on $\R^n$.

Since $\tilde{\theta}_i$ is integer-valued $\Ha^n$-a.e. there are compact sets $\tilde{K}_{ij} \subset \tilde{K}_i$ such that
\begin{equation}
\Ha^n\left( \tilde{K}_i  \backslash \bigcup_{j=1}^ \infty K_{ij} \right) = 0,
\end{equation}
such that $\tilde{\theta}_i$ is constant, equal to some $\vartheta_{ij} \in \mathbb{Z}$ on $\tilde{K}_{i j }$ for all $j \in \N$.
Moreover, for instance by mirroring $\tilde{K}_{ij}$ in one of the coordinate axes, we can assume that $\vartheta_{ij} \in \N$.

The remaining conclusions follow as in Lemma \ref{le:dividemeasure}.
\end{proof}

It follows by \cite[Section 9]{ambrosio_currents_2000} that when $T$ has a parametrization
\begin{equation}
T = \sum_{i=1}^\infty \vartheta_i (g_i)_\# \llbracket \chi_{K_i} \rrbracket,
\end{equation}
as in Lemma \ref{le:GoodParametrization}, then 
\begin{equation}
\|T\| = \sum_{i=1}^\infty \vartheta_i \lambda_{\Tan(g(K_i),.)} \Ha^n \llcorner g(K_i),
\end{equation}
where $\lambda_V$ denotes the area factor associated to a finite-dimensional Banach space $V$, defined by
\begin{equation}
\lambda_V := \frac{2^n}{\omega_n} \sup \left\{ \frac{\Ha^n(B_1)}{\Ha^n(R)}\, | \, V \supset R \supset B_1 \text{ parallelepiped } \right\},
\end{equation}
where $B_1$ is the unit ball in $V$. 
For any $n$-dimensional Banach spave $V$, $n^{-n/2} \leq \lambda_V \leq 2^k / \omega_k$.

In particular, by (\ref{eq:ComparisonToEuclidean}) and the area formula, there exists constants $c(n)$ and $C(n)$ depending only on the dimension such that for all $i \in \N$ and $A \subset K_i$, 
\begin{equation}
c(n) \|T\|(g_i(A)) \leq \lebmeas^n(A)  \leq C(n) \|T\|(g_i(A)).
\end{equation}

The main part of the following theorem was proven by Kirchheim \cite{kirchheim_rectifiable_1994}. 
The formulation in  \cite{ambrosio_rectifiable_2000} is slightly stronger.

\begin{theorem}[\cite{ambrosio_rectifiable_2000}]
Let $\mu = \theta \Ha^n \llcorner S$ be an $n$-rectifiable measure on a metric space $X$. 
Then
\begin{equation}
\Theta_n(\mu,x) = \theta(x), \qquad \text{ for } \Ha^n\text{-a.e. }x \in S.
\end{equation}
\end{theorem}

As a by-product, it follows with the area formula that for $\Ha^n$-almost every $x \in S$, in the sense of density the neighborhood of $x$ in $S$ is approximately given by the image of one of the parametrization maps.

\begin{corollary}
\label{co:ParamDens}
Let $\mu = \theta \Ha^n \llcorner S$ be a rectifiable measure on a metric space $X$, and let compact sets $K_i \subset \R^n$ and maps $g_i: K_i \to \R^n$ be given as in Lemma \ref{le:dividemeasure}. Then, for $\Ha^n$-almost every $x \in S$, there exists a unique $i_x \in \mathbb{N}$ such that $x \in g_{i_x} (K_{i_x}) $ and moreover,
\begin{equation}
\Theta_n(\mu , x ) = \Theta_n(\mu \, \llcorner g_{i_x}(K_{i_x}) , x ) = \theta(x).
\end{equation}
\end{corollary}

We will now show how to calculate the density of a set closely related to $D_x(t)$.

\begin{lemma}
\label{le:CalcCone}
Let $X$ be a separable metric space, $\mu = \theta \Ha^n \llcorner S$ be a rectifiable measure on $X$, and $f$ be a Lipschitz function defined on $S$.
Define for $x \in S$,
\begin{equation}
D_x^>(t) := \left\{x' \in X\backslash \{x\} \, | \, \frac{f(x')-f(x)}{d(x',x)} > t \right\}.
\end{equation}
Let $j:X \to Y$ be an isometric embedding of $X$ into a $w^*$-separable dual space $Y$.

Then, for $\mu$-a.e. $x \in X$, for all $t \in \R$ such that $|t| \neq |d_x^S f|$, there holds
\begin{equation}
\Theta_n( \mu \llcorner D_x^>(t) , x ) = \theta(x) \Ha^n(W_x^>(t) \cap B_1(0)),
\end{equation}
where $W_{j(x)}^>(t) \subset \Tan(j(S),j(x))$ is the cone given by
\begin{equation}
W_{j(x)}^>(t) := \left\{ w \in \Tan(j(S),j(x)) \, | \, \left\| d_{j(x)}^{j(S)} (f\circ j^{-1}) (w) \right\| > t\| w \| \right\}.
\end{equation}
\end{lemma}

\begin{proof}
Without loss of generality, we may assume that $X = \overline{S}$, and therefore, that $X$ is separable. 
Let $j: X \to Y$ be an isometric embedding of $X$ into a $w^*$-separable dual space $Y$. 
Define $\bar{f}:j(S) \to \R$ by $\bar{f} = f \circ j^{-1}$. 
Let us apply Lemma \ref{le:dividemeasure} to $j_\# \mu$ to obtain compact sets $K_i \subset \R^n$ and Lipschitz maps $g_i: \R^n \to Y$ with the properties mentioned in the Lemma. 
Let $y \in j(S)$.
By Corollary \ref{co:ParamDens}, we may assume that there exists an $i_y \in \mathbb{N}$ such that $y \in g_{i_y}(K_{i_y})$ and 
\begin{equation}
\label{eq:FullDens}
\Theta_n( j_\# \mu, y ) = \Theta_n(j_\# \mu \llcorner g_{i_y}(K_{i_y}) , y ) = \theta(x),
\end{equation}
where $x = j^{-1}(y)$. 
Without loss of generality, we may assume that $g_{i_y}^{-1}(y) = 0 \in \R^n$, and that $K_{i_y}$ has Lebesgue density $1$ at $0$.

Let $t \in \R$ be such that $|t| \neq |d_{y}^{j(S)} \bar{f}|$. 
Define the cone $C^>(t) \subset \R^n$ as
\begin{equation}
C^>(t) := wd_0 g_{i_y}^{-1} (W_y^>(t)),
\end{equation}
and let $B$ be the induced unit ball on $\R^n$,
\begin{equation}
B := wd_0 g_{i_y}^{-1} (B_1(0)).
\end{equation}
We calculate
\begin{equation}
\begin{split}
\Theta_{n} (\mu \llcorner D_x(t),x) 
	&= \lim_{r \downarrow 0} \frac{\mu \left( B_r(x) \cap D_x(t)\right) }{\omega_n r^n} \\
	&= \lim_{r \downarrow 0} \frac{j_\# \mu \left( g_{i_y}(K_{i_y}) \cap j(D_x(t)) \cap B_r(y) \right)}{\omega_n r^n} \\
	&= \lim_{r \downarrow 0} \frac{1}{\omega_n r^n} \int_{K_{i_y}} \theta(g_{i_y}(z)) \chi_{D_x(t) \cap B_r(x) }(j^{-1}(g_{i_y}(z))) 
						J_n(md_{z}g_{i_x}) dz\\
	&= \lim_{r \downarrow 0} \frac{1}{\omega_n} \int_{K_{i_y}/r} \theta( g_{i_y} ( rp ) ) 
						\chi_{ D_x(t) \cap B_r(x) } ( j^{-1}(g_{i_x}( rp) ) ) J_n(md_{rp}g_{i_x}) dp.
\end{split}
\end{equation}
Here, $\chi_A$ denotes the characteristic function of a set $A$.
For $p \in \R^n$ write $w_p := wd g_{i_y}(p)$.
Note that
\begin{equation}
\lim_{r \downarrow 0} \frac{d_X(j^{-1}(g_{i_y}(r p)),x)}{r}
= \lim_{r \downarrow 0} \frac{\|g_{i_y} (r p) - y\|_Y}{r} = \| w_y \|,
\end{equation}
and
\begin{equation}
\lim_{r \downarrow 0} \frac{f( j^{-1}(g_{i_y}(rp))) - f(x)}{d_X(j^{-1}(g_{i_y}(r p)),x)}
= \lim_{r \downarrow 0} \frac{\bar{f}(g_{i_y}(rp)) - \bar{f}(y)}{\|g_{i_y}(r p) - y\|_Y} 
  = \frac{d_{y}^{j(S)} f(w_p)}{\| w_p \|}. % > t.
\end{equation}
Consequently,
\begin{equation}
\lim_{r \downarrow 0} \chi_{D_x(t) \cap B_r(x)} (j^{-1} (g_{i_y} (rp))) = 
  \begin{cases} 
    1 & \text{ if } \| w_p \| < 1 \text{ and } | d_{y}^{j(S)} \bar{f}(w_p) | > t \| w_p \|, \\
    0 & \text{ if } \| w_p \| > 1 \text{ or } |d_{y}^{j(S)} \bar{f}(w_p)| < t \| w_p \|.
  \end{cases}
\end{equation}
Although the limit may not exist for other values of $p$, the set 
\begin{equation}
\left\{ p \in \R^n \,| \, \| w_p\| = 1 \right\} 
  \cup \left\{ p \in \R^n \, | \, |d_{y}^{j(S)} \bar{f}(w_p)| = t \| w_p \| \text{ and } \| w_p \| \leq 1 \right\}
\end{equation}
has zero $n$-dimensional Lebesgue measure when $t \neq |d_{j(x)}^{j(S)}\bar{f}|$.
Since $\theta \circ j^{-1} \circ g_{i_y}$ and $z \mapsto md_zg_{i_y}$ are continuous and the set $K_{i_y}$ has Lebesgue density one at $0$, we find 
\begin{equation}
\begin{split}
\Theta_{n}& (\mu \llcorner D_x(t),x) \\
 &= \lim_{r \downarrow 0} \frac{1}{\omega_n} \int_{K_{i_x}/r} \theta(j^{-1} ( g_{i_y} ( rp )) ) 
						\chi_{ D_x(t) \cap B_r(x) } (j^{-1}( g_{i_y}( rp)) ) J_n(md_{rp}g_{i_y}) dp\\
  &= \lim_{r \downarrow 0} \frac{1}{\omega_n} \int_{C^>(t) \cap B} \theta(x) J_n(md_0 g_{i_y}) dp \\
  &= \frac{ \theta(x) }{\omega_n} J_n(md_0g_{i_y}) \mathcal{L}^n(C^>(t) \cap B) \\
  &= \theta(x) \Ha^n(W_{j(x)}^>(t) \cap B_1(0) ).
\end{split}
\end{equation}
\end{proof}

It follows that for almost every $x$, the approximate local dilatation and the norm of the weak tangential derivative coincide:

\begin{theorem}
Let $X$ be a separable metric space, and let $\mu = \theta \Ha^n \llcorner S$ be a rectifiable measure on $X$. 
Let $f$ be a Lipschitz function defined on $S$.
Then, for $\mu$-a.e. $x \in S$, $|d_x^S f| = \ald_x f$.
\end{theorem}

\begin{proof}
By Definition \ref{de:ApDil} we need to show that
\begin{equation}
\ald_x f = \inf\left\{ t>0 \, | \, \Theta_n(\mu\llcorner D_x(t),x) = 0\right\}  = |d_x^S f|.
\end{equation}
This equality immediately follows from Lemma \ref{le:CalcCone}, since by symmetry of the cone $W_{j(x)}$, for $0 < t < | d_x^S f |$, for $\mu$-a.e. $x \in S$,
\begin{equation}
\Theta_n(\mu\llcorner D_x(t) ,x) = 2 \Theta_n(\mu\llcorner D_x^>(t),x) = 2\theta(x) \Ha^n(W_{j(x)}^>(t) \cup B_1(j(x))) > 0,
\end{equation}
while for $t > |d_x^S f|$, 
\begin{equation}
\Theta_n(\mu\llcorner D_x(t) ,x) = 2 \Theta_n(\mu\llcorner D_x^>(t),x) = 2\theta(x) \Ha^n(W_{j(x)}^>(t) \cup B_1(j(x))) = 0.
\end{equation}
\end{proof}

\subsection{The Dirichlet energy}
\label{se:TheNormEn}

\begin{definition}
\label{de:NormEn}
Let $X$ be a complete metric space, and let $T \in I_n(X)$. Denote $S = \set(X)$ and let $f$ be a Lipschitz function defined on $S$. Then we define the energy of $f$ by
\begin{equation}
E_T(f) := \int_X |d_x^S f |^2 \, d\|T\| = \int_X \ald(f)^2 \, d\|T\|
\end{equation}
and the renormalized energy of $f$ by
\begin{equation}
\mathcal{E}_T(f) := \frac{E_T(f)}{\int_X |f|^2 \, d\|T\|},
\end{equation}
if $f$ is not $\|T\|$-a.e. equal to zero, and where $\ald(f)$ is determined with respect to the $n$-rectifiable measure $\|T\|$.
\end{definition}

The definition behaves well under isometric embeddings, as expressed by the following proposition.

\begin{proposition}
\label{pr:behaviormaps}
Suppose $X$ is a metric space and $T \in I_n(X)$. Let $j:X \to Z$ be an isometric embedding of $X$ into another metric space $Z$. Let $f$ be a Lipschitz function on $Z$. Then
\begin{equation}
\mathcal{E}_T(j^{\#} f) = \mathcal{E}_{j_{\#} T} (f).
\end{equation}

\end{proposition}

\begin{proof}
Since $j:X \to Z$ is an isometric embedding, for $\|T\|$-a.e. $x\in X$
\begin{equation}
|d_x^S j^{\#} f|^2 = j^\# |d_{j(x)}^{j(S)} f | ^2
\end{equation}
and 
\begin{equation}
j_\# \|T\| = \| j_\# T \|.
\end{equation}
From this, the statement follows immediately.
\end{proof}

\begin{corollary}
If $\llbracket M \rrbracket$ is an integral current space induced by a Riemannian manifold $M$, the renormalized energy coincides with the usual Rayleigh quotient, that is, for every Lipschitz function $f$,
\begin{equation}
\mathcal{E}_{\llbracket M \rrbracket }(f) = \frac{\int_M |\nabla_M f|^2 d \mathrm{Vol}_M}{ \int_M |f|^2 d \mathrm{Vol}_M } .
\end{equation}
\end{corollary}

\begin{proof}
This follows immediately from Proposition \ref{pr:behaviormaps} since $|\nabla_M f| = |d f|$ for almost every $x \in M$.
\end{proof}

\begin{remark}
In some cases, as in Gromov's work \cite{gromov_dimension_1988}, the normalized energy is introduced with the local dilatation rather than with the \emph{approximate} local dilatation. 
The local dilatation of a function $f:X \to \mathbb{R}$ in a point $x$ is defined as
\begin{equation}
\dil_x f = \limsup_{r \downarrow 0} \sup_{y \in B_r(x)\backslash\{x\}} \frac{|f(y) - f(x)|}{d(y,x)}.
\end{equation}
We note that these may have different values. Consider for instance the following example.

For a point $x \in \R^3$, and $r>0$, let $D_r(x)$ denote the disk
\begin{equation}
D_r(x) = \{(y_1,y_2,x_3) \in \R^3 \, | \, (y_1 - x_1)^2 + (y_2 - x_2)^2 < r^2 \}.
\end{equation}
Pick a dense sequence of points $p_i \in B_1(0)$, and a sequence $r_i \downarrow 0$ as $i \to \infty$ such that the current $T$ associated to 
\begin{equation}
D_1(0) \cup \bigcup_{i=1}^\infty D_{r_i}(p_i) 
\end{equation}
is normal. Now consider the function $f(x) = x_3$ restricted to $\supp T$. Then, for almost every $x\in D_1(0)$, $\dil_x f = 1$, while $\ald_x f = 0$.
\end{remark}

\section{Completion of bounded Lipschitz functions under $W^{1,2}$-norm}
\label{se:DefSobolev}

Let $X$ be a complete metric space, and let $T \in I_n(X)$. 
By the space $L^2(\|T\|)$ we mean the usual Hilbert space of (equivalence classes of) functions on $X$, that are square-integrable (w.r.t. $\|T\|$) with inner product
\begin{equation}
(f, g)_{L^2(\|T\|)} := \int_X f g \,  d\|T\|.
\end{equation}

We will denote by $\mathcal{T}_2^*(\|T\|)$, the Banach space of (equivalence classes of) covector fields endowed with the norm
\begin{equation}
\| \psi \|_{\mathcal{T}^*_2(\|T\|)}^2 := \int_X | \psi(x) |^2_{\Tan^*(\set(T),x)} d\|T\|.
\end{equation}
We note that for a general complete metric space $X$ (not necessarily a $w^*$-separable Banach space), this is still well-defined. 

We define the function space $W^{1,2}(\|T\|)$ as the completion of the set of bounded Lipschitz functions on $\supp T$ with respect to the norm $\|.\|_{W^{1,2}}$ given by
\begin{equation}
\begin{split}
\|f \|_{W^{1,2}}^2 
& := \int_X f^2 d\|T\| + \int_X \ald(f)^2 d\|T\| \\
& = \int_X f^2 d\|T\| + \int_X |df|^2 d\|T\| \\
& = \|f\|_{L^2(\|T\|)}^2 + \| df \|_{\mathcal{T}_2^*(T)}^2.
\end{split}
\end{equation}

We would like to identify $W^{1,2}(\|T\|)$ as a subset of $L^2(\|T\|)$. 
Let $\iota : W^{1,2}(\|T\|) \to L^2(\|T\|)$ denote the natural map of the completion to $L^2(\|T\|)$. 
We will record in Corollary \ref{co:InjectivityNaturalMap} below that $\iota$ is injective. 

Although the injectivity itself can be proven more directly, we will first derive the stronger result that any function $\iota(f)$, with $f \in W^{1,2}(\|T\|)$, is tangentially differentiable $\|T\|$-a.e.. 
The main ingredient in the proof is a Poincar\'{e}-like inequality that is stated in Theorem \ref{th:PoincareLike}.
The proof is quite technical, and is therefore postponed to the appendix. 
The main idea, however, is that if $T$ is an \emph{integral} current, generically, on a small enough scale, $T$ has only small boundary, and one-dimensional slices contain a connected component over which one can integrate to obtain Poincar\'{e} inequalities on sets that are hit by these components.

Every $f$ in $W^{1,2}(\|T\|)$ can be assigned a covector field $df$ in a natural way.

\begin{definition}
Every $f \in W^{1,2}(\|T\|)$ can be represented by a Cauchy sequence $f_i$ of bounded Lipschitz functions. 
We write $d f$ for the limit of $df_i$ in $\mathcal{T}^*_2(\|T\|)$.
\end{definition}

We now state and prove the $\|T\|$-a.e. tangential differentiability of $f$. 

\begin{theorem}[Tangential differentiability of $W^{1,2}(\|T\|)$-functions]
\label{th:TangDiffSob}
Suppose $Y$ is a $w^*$-separable Banach space and $T \in I_n(Y)$. 
Then, for $\|T\|$-a.e. $p \in Y$ there exists a $w^*$-continuous and linear map $L_p$, whose restriction to $\Tan(\set(T),p)$ reduces to $df$, and a Borel set $K_p \subset Y$, such that $\Theta_n^*(\|T\| \llcorner Y \backslash K_p, p) = 0$, and 
\begin{equation}
\lim_{R \to 0} \esssup_{q \in B_R(p) \cap K_p} \frac{1}{R}|\iota(f) (q) - (\iota(f)(p)  - L_p(q - p) )| = 0.
\end{equation}
\end{theorem}

\begin{proof}
Let $T$ have the representation
\begin{equation}
T = \sum_{\ell=1}^\infty \vartheta_\ell (g_\ell)_\# \llbracket \chi_{K_\ell} \rrbracket,
\end{equation}
for $\vartheta_\ell \in \N$, Lipschitz $g_\ell:\R^n \to Y$ and $K_\ell \subset \R^n$ compact that satisfy the conclusions of Lemma \ref{le:GoodParametrization}. 
It suffices to show that the statement holds with $p = g_1(x)$ for $\lebmeas^{n}$-a.e. $x\in K_1$.

First note that for $\lebmeas^{n}$-a.e. $x \in K_1$, by a standard argument that involves the Vitali Covering Theorem (parallel to for instance \cite[Section 2.3, Theorem 1]{evans_measure_1991}),
\begin{equation}
\limsup_{R\to 0} \frac{1}{R^n} \int_{B_R(g_1(x))\backslash g_1(K_1)} |df|^2 d\|T\| = 0.
\end{equation}
By Cauchy-Schwarz also for $\lebmeas^n$-a.e. $x \in K_1$,
\begin{equation}
\lim_{R\to 0} \frac{1}{R^n} \int_{B_R(g_1(x))\backslash g_1(K_1)} |df| d\|T\| = 0.
\end{equation}
By the Lebesgue differentiation theorem, for $\lebmeas^n$-a.e. $x \in K_1$,
\begin{equation}
\limsup_{r \to 0} \frac{1}{r^n} \int_{K_1 \cap B_r(x)} | (g_1^\# df) - (g_1^\# df)(x) |  d\lebmeas^n = 0,
\end{equation}
where $(g_1^\# df)(x) = d_{g_1(x)} f \circ wd_x g_1 \in (\R^n)^*$ is the pull-back of $df$ under $g_1$ evaluated at $x$, and the norm in the integrand is the dual Euclidean norm induced by the Euclidean norm on $\R^n$. 

At such a point $x \in K_1$, that is also a Lebesgue point for $\iota(f)$, let $L_{g_1(x)}:Y \to \R$ be a linear and $w^*$-continuous extension of $df$ defined on $\Tan(\set(T),g_1(x))$. 
By the continuity of the map $z \to wd_z g_1$ on $K_1$, guaranteed by the last item in Lemma \ref{le:GoodParametrization}, it immediately follows that also
\begin{equation}
\limsup_{r \to 0} \frac{1}{r^n} \int_{K_1 \cap B_r(x)} | g_1^\# (L_{g_1(x)}|_{\Tan(\set{T}, g_1(.))}) - (g_1^\# df )(x)| d \lebmeas^n = 0.
\end{equation} 

Since $g_1$ satisfies (\ref{eq:ComparisonToEuclidean}) on $K_1$, it also follows that, with $R = (3 + \sqrt{n})r$,
\begin{equation}
\limsup_{r \to 0} \frac{1}{r^n} \int_{B_R(g_1(x)) \cap g_1(K_1)} | (L_{g_1(x)} |_{\Tan(\set T, .)}) - df  | d \|T\| = 0.
\end{equation}

We will use the notation
\begin{equation}
(h)_G := \frac{1}{\lebmeas^n(G) }\int_G h d\lebmeas^n,
\end{equation}
for the average of an integrable function $f : h \to \R$ on a Borel set $G$.

Let $\epsilon > 0$. 
It follows from the above and the Poincar\'{e}-like inequality in Theorem \ref{th:PoincareLike} that for $\lebmeas^n$-a.e. $x \in K_1$, there exists a function $r \mapsto \delta(r)$ such that $\delta(r) \downarrow 0$ as $r \to 0$, and for small enough $r$ there exists a $G_r \subset Q_r(x)$ with $\Ha^n(G_r) \geq (1- \delta(r)) (2r)^n$, such that for all $i= 1, 2, \dots$, with $R = (3 + \sqrt{n}) r$,
\begin{equation}
\begin{split}
\frac{2}{(2r)^n r} \int_{G_r} & | (\iota(f_i) - L_{g_1(x)})\circ g_1 - ((\iota(f_i) - L_{g_1(x)}) \circ g_1)_{G_r} | d\lebmeas^n \\
& \leq \frac{C(n) r}{r \|T\|(B_R(g_1(x)))} \int_{B_R(g_1(x))} |df_i - (L_{g_1(x)}|_{\Tan(\set T, .)})| d\|T\| \\
&  \leq \frac{\tilde{C}(n)}{r^{n}} \int_{B_R(g_1(x))} |df_i -df| d\|T\| \\
&\quad + \frac{\tilde{C}(n)}{r^{n}} \int_{B_R(g_1(x)) \cap g_1(K_1)} |df - (L_{g_1(x)}|_{\Tan(\set T, .)}) | d\|T\| \\
& \quad + \frac{\tilde{C}(n)}{r^{n}} \int_{B_R(g_1(x)) \backslash g_1(K_1)} |d (L_{g_1(x)}|_{\Tan(\set T, .)}) | d\|T\|\\
&\quad + \frac{\tilde{C}(n)}{r^{n}} \int_{B_R(g_1(x)) \backslash g_1(K_1)} |df| d\|T\|\\
& < \frac{\tilde{C}(n)}{r^{n}} \int_{B_R(g_1(x))} | df_i - df | d\| T \|+ \epsilon.
\end{split}
\end{equation}
Since $f_i \to \iota(f)$ in $L^2(\|T\|)$, and $df_i \to df$ in $\mathcal{T}^*_2(\|T\|)$, it holds in particular that
\begin{equation}
\limsup_{r \to 0} \frac{2}{(2r)^n r} \int_{G_r} |(\iota(f) - L_{g_1(x)}) \circ g_1 - ((\iota(f) - L_{g_1(x)} )\circ g_1)_{G_r}| d\lebmeas^n = 0.
\end{equation}
Set $r_j = 2^{-j}$. 
For large enough $j$, that is for small enough $\delta(r_j)$, the previous formula implies that
\begin{equation}
| ((\iota(f) - L_{g_1(x)}) \circ g_1)_{G_{r_{j+1}}} - ((\iota(f) - L_{g_1(x)}) \circ g_1)_{G_{r_j}}| \leq \omega(r_{j}) r_j ,
\end{equation}
for a function $\omega: (0,\infty) \to (0,\infty)$ with $\omega(r) \downarrow 0$ as $r \downarrow 0$.
By a telescoping argument, it holds that also
\begin{equation}
\frac{ ((\iota(f) - L_x) \circ g_1)_{G_r} - \iota(f)(g_1(x))}{r} \to 0, 
\end{equation}
so that $\lim_{r' \to 0} h(r') = 0$, with
\begin{equation}
h(r') := \sup_{0 < r \leq r' } \frac{1}{(2r)^{n} r} \int_{G_r} | \iota(f) \circ g_1 - \iota(f)(g_1(x)) - L_{g_1(x)}(g_1(.) - g_1(x)) | d \lebmeas^n = 0.
\end{equation}
We finally construct the set $K_{g_1(x)}$ mentioned in the lemma. 
Let still $r_j = 2^{-j}$. 
We may select compact sets $K_{r_j} \subset G_{r_j}$, such that for all $y \in K_{r_j}$,
\begin{equation}
\frac{1}{r_j} |\iota(f)(g_1(y)) - \iota(f)(g_1(x)) - L_{g_1(x)}(g_1(y) - g_1(x)) | < \sqrt{h(r_j)},
\end{equation}
while $\lebmeas^n(G_{r_j} \backslash K_{r_j}) \leq \sqrt{h(r_j)}(2r_j)^n$.
We define 
\begin{equation}
K_{g_1(x)} := g_1\left( \bigcup_j K_{r_j} \cap ( Q_{r_j}(x) \backslash Q_{r_{j+1}} (x) ) \right).
\end{equation}
\end{proof}

\begin{corollary}
\label{co:InjectivityNaturalMap}
Let $X$ be a complete metric space, and let $T \in I_n(X)$. 
Then, the natural map $\iota: W^{1,2}(\|T\|) \to L^2(\|T\|)$ is injective.
\end{corollary}

\begin{proof}
Without loss of generality, we may assume that $X$ is a $w^*$-separable Banach space. 
If $\iota(f) = 0$, the previous lemma immediately implies that $df = 0$ in $\|T\|$-a.e. $p \in X$. 
\end{proof}

From now on, we will occasionally identify $W^{1,2}(\|T\|)$ with its image under  the map $\iota$ in $L^2(\|T\|)$.

\begin{corollary}
\label{co:GeneralEqQuant}
Let $X$ be a complete metric space and let $T \in I_n(X)$. For every $f \in W^{1,2}(\|T\|)$,
$|df| = \ald(f)$, $\|T\|$-a.e..
\end{corollary}

This corollary follows from Theorem \ref{th:TangDiffSob} by an argument as in Section \ref{se:EqQuant}.

\section{Lower-semicontinuity of the energy}
\label{se:LowSemEnergy}

By John's Ellipsoid Theorem (see also Corollary \ref{co:JohnEllipsoid}), on almost every dual to the approximate tangent space, there exists a Hilbert space norm $\| . \|_J$, the unit ball of which is the John ellipsoid inside the unit ball with respect to the norm on the dual tangent space, such that
\begin{equation}
\frac{1}{\sqrt{n}} \| v \|_J \leq | v | \leq \| v \|_J, \qquad v \in \Tan^*(S,x).
\end{equation}
We define the energy
\begin{equation}
E_T^J(f) = \int_X \|d f\|_J^2 d\|T\|.
\end{equation}
Clearly, $E_T^J$ is a convex quadratic form, moreover
\begin{equation}
\frac{1}{n} E_T^J(f) \leq E_T(f) = \int_X | d f|^2 d\|T\| \leq E_T^J(f).
\end{equation}

We define $W^{1,2}_J(\|T\|)$ as the completion of the set of bounded Lipschitz functions under the norm $\|.\|_{W^{1,2}_J}$ given by
\begin{equation}
\| f \|_{W^{1,2}_J}^2 := \|f\|^2_{L^{2}(\|T\|)} + E_T^J(f).
\end{equation} 
Note that $W^{1,2}_J(\|T\|)$ is a Hilbert space, and
\begin{equation}
\frac{1}{\sqrt{n}} \| . \|_{W_J^{1,2}} \leq \| . \|_{W^{1,2}} \leq  \| . \|_{W_J^{1,2}}.
\end{equation}
\begin{theorem}[lower semicontinuity of the energy]
\label{th:LowSCEnergy}
Let $X$ be a complete metric space and let $T \in I_n(X)$.
Let $f_j$ be a uniformly bounded sequence in $W^{1,2}(\|T\|)$, such that $\iota(f_j) \rightharpoonup f \in L^2(\|T\|)$ weakly in $L^2(\|T\|)$. 
Then, in fact, there is a unique $\tilde{f} \in W^{1,2}(\|T\|)$ such that $\iota(\tilde{f}) = f$. Moreover,
\begin{equation}
\liminf_{j \to \infty} E_T(f_j) = \liminf_{j \to \infty} \int_X | df_j |^2 d\| T \| \geq \int_X |d \tilde{f}|^2 d\|T\| = E_T(\tilde{f}).
\end{equation}
\end{theorem}

\begin{proof}
Note that since $E_T(f_j)$ is uniformly bounded, $E_T^J(f_j)$ is uniformly bounded as well. 
Therefore, as $W_J^{1,2}(\|T\|)$ is a Hilbert space, a subsequence of $f_j$ weakly converges to an element $\tilde{f}$ in $W^{1,2}_J(\|T\|)$.
By Mazur's lemma, a sequence of convex combinations $g_j$ of the $f_j$'s converges strongly in $W^{1,2}_J$ to $\tilde{f}$.
More precisely, there is a function $N:\mathbb{N} \to \mathbb{N}$ and nonnegative coefficients $\alpha(k)_\ell$ ($\ell = k, \dots, N(k))$, with 
\begin{equation}
\sum_{\ell = k}^{N(k)} \alpha(k)_\ell = 1, 
\end{equation}
such that the sequence $g_k$ defined by
\begin{equation}
g_k := \sum_{\ell = k}^{N(k)} \alpha(k)_\ell f_\ell,
\end{equation}
converges strongly to the function $\tilde{f} \in W_J^{1,2}(\|T\|)$.
In particular, $\iota(\tilde{f}) = f$.
The uniqueness of $\tilde{f}$ follows from the injectivity shown in Corollary \ref{co:InjectivityNaturalMap}.

Moreover, $g_j$ also converges strongly to $\tilde{f}$ in $W^{1,2}(\|T\|)$, and since $E_T(.)$ is convex,
\begin{equation}
\begin{split}
\liminf_{j\to\infty} \int_X |df_j|^2 d\|T\|
&\geq \liminf_{j\to\infty} \int_X |d g_j|^2 \|T\| \\
&= \int_X |d \tilde{f}|^2 \|T\|. 
\end{split}
\end{equation}
\end{proof}

\subsection{Approximate local dilatation is minimal relaxed gradient}

We will now compare the definition of the approximate local dilatation, with that of the minimal relaxed gradient as introduced by Ambrosio, Gigli and Savar\'{e} in \cite[Definition 4.2]{ambrosio_calculus_2014}. 
In fact, we will show in Theorem \ref{th:ApdilIsMinRelaxGrad} below that for integral currents, the two quantities coincide $\|T\|$-a.e.. It follows that the energy $E_T$ coincides with the Cheeger energy for $(X , d_X, \|T\|)$ \cite{ambrosio_calculus_2014,cheeger_differentiability_1999}.

Let us first recall the definition of a minimal relaxed gradient, and simplify it for the case at hand. Recall that for a function $f : X \to \R$ on a metric space $X$, we mean by $\dil_x f$ the quantity
\begin{equation}
\dil_x f := \limsup_{r \to 0} \, \sup_{y \in B_r(x) \backslash \{x\}} \frac{|f(y) - f(x)|}{d_X(x,y)}.
\end{equation}

\begin{definition}[Specified version of {\cite[Definition 4.2]{ambrosio_calculus_2014}}]
We say that $G \in L^2(\|T\|)$ is a relaxed gradient of $f \in L^2(\|T\|)$, if there exists a sequence $f_1, f_2, \dots$ of Lipschitz functions in $L^2(\|T\|)$ such that 
\begin{enumerate}
\item $f_n \to f$ in $L^2(\|T\|)$ and $\dil f_n$ weakly to $\tilde{G} \in L^2(\|T\|)$,
\item $\tilde{G} \leq G$ $\|T\|$-a.e. in $X$.
\end{enumerate}
\end{definition}

A function $G \in L^2(\|T\|)$ is called the minimal relaxed gradient if its $L^2(\|T\|)$ norm is minimal among relaxed gradients. The minimal relaxed gradient is guaranteed to exist \cite{ambrosio_calculus_2014}.

\begin{theorem}
\label{th:ApdilIsMinRelaxGrad}
Let $X$ be a complete metric space, and let $T \in I_n(X)$. 
Let $f \in L^2(\|T\|)$.
Then $f$ has a relaxed gradient in the sense of Ambrosio-Gigli-Savar\'{e} \cite[Definition 4.2]{ambrosio_calculus_2014}, if and only if $f \in W^{1,2}(\|T\|)$.
Moreover, the minimal relaxed gradient equals both $|df|$ and $\ald(f)$, $\|T\|$-a.e..
\end{theorem}

\begin{proof}
Without loss of generality, we may assume that $X$ is $w^*$-separable.

If $f \in W^{1,2}(\|T\|)$ there is a sequence of bounded Lipschitz functions $\tilde{f}_i$ such that $\tilde{f}_i \to f$ in $L^2(\|T\|)$ and $d \tilde{f}_i \to df$ in $\mathcal{T}^*_2(\|T\|)$. 
We may use Lemma \ref{le:dividecurrent} below, combined with a construction as in the proof of Theorem \ref{th:SemContMinMaxFlat}, to approximate every $\tilde{f}_i$ by a bounded Lipschitz function $f_i$, such that in fact $f_i \to f$ and $\dil f_i \to |df|$ in $L^2(\|T\|)$. 
This shows that $|df|$ is a relaxed gradient of $f$. 
Hence $f$ also has a minimal relaxed gradient.

Suppose $f \in L^2(\|T\|)$ has a relaxed gradient. 
Then it also has a minimal relaxed gradient, $G$ say.
According to \cite[Lemma 4.3]{ambrosio_calculus_2014}, there is a sequence of bounded Lipschitz functions $f_j \in L^2(\|T\|)$ such that $f_j \to f$ strongly in $L^2(\|T\|)$ and $\dil f_j \to G$ in $L^2(\|T\|)$.
In particular, $\dil f_j$ is uniformly bounded in $L^2(\|T\|)$, so that $E_T(f_j)$ is uniformly bounded, and therefore, by the semicontinuity shown in Theorem \ref{th:LowSCEnergy}, $f \in W^{1,2}(\|T\|)$, and
\begin{equation}
\int_X |df|^2 d\|T\| \leq \liminf_{j \to \infty} \int_X |df_j|^2 d\|T\| = \int_X G^2 d \|T\|.
\end{equation}
By the minimality of $G$, $G = |df|$, $\|T\|$-a.e.. By Corollary \ref{co:GeneralEqQuant}, also $|df| = \ald(f)$ $\|T\|$-a.e..
\end{proof}

\section{Semicontinuity of min-max values}
\label{se:minmax}

We define
\begin{equation}
\mathcal{V}(\|T\|) = \left\{f \in W^{1,2}(\|T\|) \, | \, \int_X f \, d \| T \| = 0 \right\},
\end{equation}
and the first eigenvalue of a current $T$, $\lambda_1$, as the smallest critical point of $\mathcal{E}_T$. 
That is,
\begin{equation}
\label{eq:InfEn}
\lambda_1 = \inf_{f \in \mathcal{V}(\|T\|)} \mathcal{E}_T(f).
\end{equation}
We also define higher order min-max values $\lambda_k$
\begin{equation}
\label{eq:minmax}
\lambda_k(T) := \inf_{\substack{\{\phi_1 , \dots , \phi_{k}\} \subset \mathcal{V}(\|T\|) \\ L^2(\| T\| )-\text{orthonormal} }} \sup_{i=1,\dots,k} \mathcal{E}_T(\phi_i).
\end{equation}
Recall that this corresponds to the Rayleigh quotient on Riemannian manifolds. 
Semicontinuity of these values holds in the general case as well. Before we prove this, we first state a helpful lemma.

\begin{lemma}[dividing a rectifiable set, the domain of Lipschitz functions]
\label{le:dividecurrent}
Let $Y$ be a $w^*$-separable dual space, let $T \in I_n(Y)$ and let $f_1, \dots, f_k \in \Lip(S)$, where $S=\set(T)$. 
Moreover, let $\epsilon > 0$. 
Then there exist compact sets $K_i \subset \mathbb{R}^n$ and bi-Lipschitz maps $g_i: K_i \to Y$ such that the $g_i$ and $K_i$ have all the properties of Lemma \ref{le:dividemeasure} with respect to $\mu = \|T\|$, and additionally, for every $i \in \mathbb{N}$, $m = 1, \dots, k$, the functions $x \mapsto |d_x^S f_m(x)|$ restricted to $g_i(K_i)$ are continuous, and for moduli of continuity $\tilde{\omega}_i$, and all $y,z \in K_i$,
\begin{multline}
\label{eq:unimod}
\left| (f_m\circ g_i)(y) - (f_m\circ g_i)(z) - d_z(f_m\circ g_i)(y-z) \right| \\ \leq \tilde{\omega}_i(\| g_i(y) - g_i(z)\|) \, \|g_i(y) - g_i(z)\|, 
\end{multline}
and therefore in particular, it can be assured that, with $c_m^i := \Lip(f_m|_{g_i(K_i)})$,
\begin{equation}
\label{eq:oscbound}
(c_m^i)^2 - \epsilon \leq |d_x^S f_m|^2 \leq (c_m^i)^2, \qquad \text{ for all } x \in g_i(K_i).
\end{equation}
\end{lemma}

\begin{proof}
After obtaining the $K_i$ and $g_i$ as in Lemma \ref{le:dividemeasure}, we can apply Lusin's Theorem repeatedly to find compact sets $K_{ij}\subset K_i$ that for fixed $i$ cover the sets $K_i$ up to a set of measure zero, that is
\begin{equation}
\Ha^n\left( K_i \backslash \bigcup_{j=1}^\infty K_{ij} \right) = 0,
\end{equation}  
and such that $x \mapsto |d_x^S f_m|$ is continuous on $g_i(K_{ij})$.

Note that on the other hand, for all $z \in K_i$,
\begin{equation}
\label{eq:DerivUnif}
\lim_{K_i \ni y \to z} \left( \frac{(f_m\circ g_i)(y) - (f_m\circ g_i)(z)}{\|g_i(y)-g_i(z)\|} - \frac{d_z(f_m\circ g_i)(y-z)}{\|g_i(y)-g_i(z)\|} \right) = 0,
\end{equation}
so that by Egorov's Theorem, there are compact sets $\tilde{K}_{ij} \subset K_i$, $j=1, 2, \dots$, again covering $K_i$ up to measure zero, such that (\ref{eq:unimod}) holds for $\tilde{\omega}_i$ replaced by a modulus of continuity $\omega_{ij}$. After taking intersections and reindexing, we have shown the first part of the lemma.

Now let $\epsilon > 0$, let $y, z \in K_i$. From (\ref{eq:unimod}), we find
\begin{equation}
|(f_m\circ g_i)(y) - (f_m\circ g_i)(x)| \leq |d_z(f_m\circ g_i)(y-z)| + \tilde{\omega}_i( \|g_i(y)-g_i(z)\|)\|g_i(y)-g_i(z)\|.
\end{equation}
Since $d_z(f_m \circ g_i) = d_{g_i(z)}^S f_m \circ wd_zg_i$ and $\|wd_z g_i(y-z)\| = md_z g_i(y-z)$, we find also using (\ref{eq:UnifApprDist}),
\begin{equation}
\begin{split}
|(f_m\circ g_i)(y) - (f_m\circ g_i)(z)| & \leq |d_{g_i(z)}^S f_m| \|wd_z g_i (y-z)\|  \\
&\qquad  + \tilde{\omega}_i( \|g_i(y)-g_i(z)\|)  \|g_i(y)-g_i(z)\| \\
& \leq \Big( |d_{g(z)}^S f_m| (1 + \omega_i(\|g_i(y) - g_i(z)\|)) \\ 
&\qquad+ \tilde{\omega}_i ( \|g_i(y)-g_i(z)\|) \Big) \|g_i(y)-g_i(z)\|.
\end{split}
\end{equation}
Again replacing each set $K_i$ by a countable collection of compact subsets, with small enough diameters, the union of which covers $K_i$ up to a set of measure zero, we ensure (\ref{eq:oscbound}).
\end{proof}

\begin{theorem}
\label{th:SemContMinMaxFlat}
Suppose $Y$ is a $w^*$-separable dual space and $T_i \in I_n(Y)$ converge in the flat distance on $Y$ to a current $T \in I_n(Y)$ such that $\Mass(T_i) \to \Mass(T)$ as $i \to \infty$. 
Then
\begin{equation}
\limsup_{i \to \infty} \lambda_k( T_i ) \leq \lambda_k( T ).
\end{equation}
\end{theorem}

\begin{proof}
Let  $\sigma > 0$ and choose bounded Lipschitz functions $f_1, \dots, f_k$ such that they are $L^2(\|T\|)$-orthonormal and for $j = 1, \dots, k$
\begin{equation}
\mathcal{E}_T(f_j) \leq \mathcal{E}_T(f_k) 
	\leq \lambda_k(T) + \sigma.
\end{equation}
Let $\epsilon > 0$ and apply Lemma \ref{le:dividecurrent} to the current $T$ and the functions $f_j$, to obtain functions $g_\ell$ and sets $K_\ell$ ($\ell = 1, 2, \dots$) as in the Lemma.

Select $N$ large enough such that
\begin{equation}
\label{eq:KlCovEps}
\|T\| \left( \set T \backslash \bigcup_{\ell=1}^N g_\ell(K_\ell) \right) < \epsilon.
\end{equation}
The sets $g_\ell(K_\ell)$ are compact and disjoint, so the minimal distance $\delta$ between the sets is positive,
\begin{equation}
\delta := \min_{\substack{ i,j = 1, \dots N \\ i < j} } \dist\left(g_i(K_i), g_j(K_j) \right) > 0.
\end{equation}
Define the open sets $U_j\subset Y$ as the $\delta/10$ neighborhoods of $K_j$. 

We claim that we can extend the functions $f_j$ to bounded Lipschitz functions $\hat{f}_j$ on the whole of $Y$, so that $\Lip(\hat{f}_j) \leq 2 \Lip( f_j )$, $\sup |\hat{f}_j| \leq 2 \sup |f_j| =: M_j$, and
\begin{equation}
\label{eq:BoundLipRestr}
\Lip\left(\hat{f}_j|_{U_j}\right) \leq \Lip\left(f_j|_{g_\ell(K_\ell)}\right) =: c_j^\ell.
\end{equation}
This can be done as follows. First, we define 
\begin{equation}
\bar{f}_j^\ell(x) := \inf_{a\in g_\ell(K_\ell) } f_j(a) + c_j^\ell \| a - x \|,\qquad x \in U_\ell,
\end{equation}
and if necessary, we truncate $\hat{f}_j^\ell := (\bar{f}_j^\ell \land M_j ) \lor (-M_j)$. 
Note that $\Lip(\hat{f}_j^\ell) \leq c_j^\ell$.
Subsequently, we consider the functions $\hat{f}_j: \cup_\ell U_\ell \to \R$, given by
\begin{equation}
\hat{f}_j(x) = \hat{f}_j^\ell(x), \qquad \text{ if } x \in U_\ell.
\end{equation}
Note that the Lipschitz constant of $\hat{f}_j$ is less than $2 \Lip(f_j)$. Indeed, if $x\in U_{\ell_1}$, $y \in U_{\ell_2}$, $\ell_1 \neq \ell_2$, there exist $x_0 \in g_{\ell_1}(K_{\ell_1})$ and $y_0 \in g_{\ell_2}(K_{\ell_2})$ such that $|x-x_0| < \delta/5$ and $|y-y_0| < \delta/5$ and therefore
\begin{equation}
\begin{split}
|\hat{f}_j(x) - \hat{f}_j(y)| &\leq |\hat{f}_j(x) - \hat{f}_j(x_0)| + |\hat{f}_j(x_0) - \hat{f}_j(y_0)| + |\hat{f}_j(y_0) - \hat{f}_j(y)| \\
&\leq c_j^{\ell_1} |x-x_0| + \Lip(f_j) |x_0 - y_0| + c_j^{\ell_2} |y_0 - y| \\
&\leq \Lip(f_j)\left( \frac{1}{4} |x-y| + \frac{5}{4}  |x - y| + \frac{1}{4} |x-y| \right) \\
&< 2 \Lip(f_j) |x-y|.
\end{split}
\end{equation}
Consequently, we can extend the functions $\hat{f}_j$ to Lipschitz functions on the whole of $Y$ as claimed.

As explained in Lemma \ref{le:weakcurweakmeas}, by the portmanteau theorem (cf. \cite{klenke_probability_2007}), $\|T_i\| \rightharpoonup \|T\|$ weakly as measures on $Y$. 
As the $\hat{f}_j$ are bounded and Lipschitz, this implies that
\begin{subequations}
\label{eq:asymptorth}
\begin{alignat}{2}
\lim_{i \to \infty} \int_Y \hat{f}_{j_1} \hat{f}_{j_2} d \|T_i\| 
	 &= \lim_{i \to \infty} \int_Y \hat{f}_{j_1} \hat{f}_{j_2} d\| T \| 
	 &&= \delta_{j_1,j_2},\\
\lim_{i \to \infty} \int_Y \hat{f}_{j_1} d\|T_i\|
	&=  \int_Y \hat{f}_{j_1} d\|T\| 
	&&= 0,
\end{alignat}
\end{subequations}
for all $j_1, j_2 \in \{ 1, \dots, k\}$.

We can restrict the functions $\hat{f}_j$ to $\set(T_i)$, subtract the average, 
\begin{equation}
f_j^i := \hat{f}_j - \frac{1}{\Mass(T_i)}\int_Y \hat{f}_j d\|T_i\|,
\end{equation}
and then apply Gram-Schmidt to obtain $\psi^i_1, \dots, \psi^i_k$, an $L^2(\|T_i\|)$-orthonormal system of bounded Lipschitz functions.
That is,
\begin{align}
\psi_1^i &:= \frac{f_1^i}{\| f_1^i \|_{L^2(\|T_i\|)}},\\
\nonumber \hat{\psi}_{j+1}^i &:= f_{j+1}^i - \left(f_{j+1}^i,  \psi_1^i \right)_{L^2(\|T_i\|)} \psi_1^i \\
&\qquad					- \left(f_{j+1}^i,  \psi_2^i \right)_{L^2(\|T_i\|)} \psi_2^i 
					- \dots 
					- \left(f_{j+1}^i,  \psi_j^i \right)_{L^2(\|T_i\|)} \psi_j^i,\\
\psi_{j+1}^i &:= \frac{\hat{\psi}_{j+1}^i}{\| \hat{\psi}_{j+1}^i \|_{L^2(\|T_i\|)}}, \qquad j = 1, 2, \dots, k-1.
\end{align}
Note that in particular, $\int_Y \psi_j^i \, d\|T_i\| = 0$, so that $\psi_j^i \in \mathcal{V}(\|T_i\|)$. 
As a result of the Gram-Schmidt algorithm, the functions $\psi_{j_1}^i$ satisfy, for $j_1 = 1, \dots, k$,
\begin{equation}
\psi_{j_1}^i = \sum_{j_2 = 1}^k a^i_{j_1 j_2} \hat{f}_{j_2} + b^i_{j_1},
\end{equation} 
where $a^i_{j_1 j_2}$ and $b^i_{j_1}$ are constants, $j_1,j_2 \in \{1,\dots,k\}$, which by (\ref{eq:asymptorth}) satisfy
\begin{subequations}
\label{eq:CoeffToDelta}
\begin{align}
\lim_{i \to \infty} a^i_{j_1 j_2} &= \delta_{j_1 j_2}, \\
\lim_{i \to \infty} b^i_{j_1} &= 0.
\end{align}
\end{subequations}
Observe that this implies that
\begin{equation}
\label{eq:Switchpsif}
\lim_{i \to \infty} \int_Y \left| |d\psi_j^i|^2 - |d \hat{f}_j|^2 \right| d\|T_i\| = 0.
\end{equation}
By (\ref{eq:Switchpsif}), (\ref{eq:KlCovEps}) and the fact that $\Lip(\hat{f}_j) < 2 \Lip(f_j)$,
\begin{equation}
\label{eq:EstSecTermPsi}
\begin{split}
\limsup_{i \to \infty}  \int_Y |d\psi_j^i|^2 d\|T_i\| 
& \leq \limsup_{i \to \infty} \sum_{\ell=1}^N \int_{U_\ell} |d \psi_j^i|^2 d\|T_i\| 
	+ \limsup_{i \to \infty} \int_{(\cup_\ell U_\ell)^c} |d \psi_j^i|^2 d\|T_i\|
\\ 
& \leq \limsup_{i \to \infty} \sum_{\ell=1}^N \int_{U_\ell} |d \psi^i_j|^2 d\|T_i\| +
 2 \epsilon \sup_{j} \Lip(f_j)^2.
\end{split}
\end{equation}
For the first term, we have by (\ref{eq:Switchpsif}), (\ref{eq:BoundLipRestr}), and the bound (\ref{eq:oscbound}) from the application of Lemma \ref{le:dividecurrent},
\begin{equation}
\label{eq:EstFirstTermPsi}
\begin{split}
\limsup_{i \to \infty} \sum_{\ell=1}^N \int_{U_\ell} |d \psi_j^i|^2 d\|T_i\| 
	& \leq \limsup_{i \to \infty} \sum_{\ell=1}^N \int_{U_\ell} |d \hat{f}_j |^2 d\|T_i\| \\
	& \leq \limsup_{i \to \infty} \sum_{\ell=1}^N  (c_j^\ell)^2 \| T_i \|(U_\ell) \\
	& \leq \sum_{\ell = 1}^N \int_{U_\ell} |d f_j|^2 d\|T\| + \epsilon \mathbf{M}(T) \\
	& \leq \lambda_k(T) + \sigma + \epsilon \mathbf{M}(T).
\end{split}
\end{equation}
Since $\{\psi_j^i\}_{j=1}^k \subset \mathcal{V}(\|T_i\|)$ are $L^2(\|T_i\|)$-orthonormal, we conclude from (\ref{eq:EstSecTermPsi}) and (\ref{eq:EstFirstTermPsi}) that
\begin{equation}
\limsup_{i \to \infty} \lambda_k(T_i) \leq \lambda_k(T) + \sigma + \epsilon \mathbf{M}(T) + 2 \epsilon \sup_j \Lip(f_j).
\end{equation}
Because $\sigma$ and $\epsilon$ were arbitrary, and the $f_j$ do not depend on $\epsilon$, this implies the theorem.
\end{proof}

\section{Semicontinuity for min-max values under intrinsic flat convergence}
\label{se:SemInfEn}

In this section we define the infimum of the normalized energy $\lambda_1$ and the other min-max values $\lambda_k$ for integral current spaces, and show that they are semicontinuous under intrinsic flat convergence if the mass converges as well.

\subsection{Min-max values for integral current spaces}

We first define the infimum of the normalized energy $\lambda_1$ and the min-max values $\lambda_k$ following the definitions in  (\ref{eq:InfEn}) and (\ref{eq:minmax}).

\begin{definition}
Given a nonzero integral current space $M=(X,d,T)$ we define $\lambda_1(M)$ as the infimum of the normalized energy
\begin{equation}
\lambda_1(M) = \inf_{ f \in \mathcal{V}(\|T\|) } \mathcal{E}_T(f),
\end{equation}
and the min-max values $\lambda_k(M)$
\begin{equation}
\lambda_k(M) := \inf_{\substack{\{\phi_1 , \dots , \phi_{k}\} \subset \mathcal{V}(\|T\|) \\ L^2(\| T\| )-\text{orthonormal} }} \sup_{i=1,\dots,k} \mathcal{E}_T(\phi_i),
\end{equation}
with $\mathcal{E}_T$ as in Definition \ref{de:NormEn}.
\end{definition}

Observe that when $M$ is induced by an oriented Riemannian manifold, $\lambda_k(M)$ is its $k$th Neumann eigenvalue.

The intrinsic flat distance between two integral current spaces is zero if and only if there exists a current preserving isometry between the two spaces. The following Lemma states that in that case, the min-max values of the two spaces coincide.

\begin{lemma}
Let $M_1 = (X_1, d_1, T_1)$ and $M_2 = (X_2, d_2, T_2)$ be integral current spaces and let $\phi: X_1 \to X_2$ be a current-preserving isometry. That is, besides being an isometry from $X_1$ to $X_2$, $\phi$ also satisfies $\phi_\# T_1 = T_2$. Then for all $k=1, 2, \dots$, it holds that $\lambda_k(M_1) = \lambda_k(M_2)$.
\end{lemma}
\begin{proof}
This follows immediately from Proposition \ref{pr:behaviormaps}.
\end{proof}

\subsection{Semicontinuity for min-max values under intrinsic flat convergence}

Theorem \ref{th:SemContMinMaxFlat} immediately implies semicontinuity of the min-max values $\lambda_k$ under intrinsic flat convergence when the total mass is conserved as well.

\begin{theorem}[Upper-semicontinuity of min-max values]
\label{Th:SemContFirst}
Let $(X_i,d_i,T_i)$, (with $i=1,2,\dots$), be a sequence of integral current spaces converging in the intrinsic flat distance to a nonzero integral current space $(X,d,T)$ such that additionally, $\mathbf{M}(T_i) \to \mathbf{M}(T)$ as $i \to \infty$. 
Then one has semicontinuity of the min-max values
\begin{equation}
\limsup_{i \to \infty} \lambda_k(T_i) \leq \lambda_k(T).
\end{equation}

\end{theorem}

\begin{proof}
Since the integral current spaces $(X_i,d_i,T_i)$ converge in the intrinsic flat sense to $(X,d,T)$, by Theorem \ref{th:embcommonspace} established by Sormani and Wenger, and the Kuratowski embedding, there exist a $w^*$-separable dual space $Y$ and isometric embeddings $\phi_i: \overline{X_i} \to  Y$, $\phi: \overline{X} \to Y$ such that $(\phi_i)_\# T_i \to (\phi)_\# T$ in the flat distance in $Y$. Then we may apply Theorem \ref{th:SemContMinMaxFlat}.
\end{proof}

 \section{Infinitesimally Hilbertian integral currents}
 \label{se:InfHilb}
 
 In this section we additionally assume that the currents involved are infinitesimally Hilbertian, that is, that (almost everywhere) the norm on the tangent spaces to their rectifiable set is induced by an inner product. 
 This assumption is similar to the one made by Cheeger and Colding \cite{cheeger_structure_2000}, and ensures that there is a quadratic form associated to the energy.
 
 More precisely, for a metric space $X$ and an integral current $T \in I_n(X)$, we assume that for $\|T\|$-a.e. $x \in X$, $\Tan(\set{T},x)$ is an inner product space. Denote the inner product on the dual space to $\Tan(\set{T},x)$ by $g_x(.,.)$.
 
\begin{theorem}
\label{th:DefLapl}
Let $X$ be a complete metric space and let $T \in I_n(X)$. Then the quadratic form 
\[
Q_T(f,g) := \int_X g_x( d_x^S f, d_x^S g ) \, d\| T \|,
\]
with form-domain $W^{1,2}(\|T\|) \subset L^2(\|T\|)$ is closed. Consequently, there is a unique associated (unbounded) nonnegative self-adjoint operator $-\Delta_T$ on $L^2(\|T\|)$ that satisfies $Q_T(\phi,\phi) = (\phi,- \Delta_T \phi)$ for every $\phi \in \mathcal{D}(\Delta_T)$.
\end{theorem}
 
\begin{proof}
We may interpret $W^{1,2}(\|T\|)$ as a subset of $L^2(\|T\|)$ by the injectivity of the natural map $\iota:W^{1,2}(\|T\|) \to L^2(\|T\|)$, shown in Corollary \ref{co:InjectivityNaturalMap}. The closedness of $Q_T$ is then immediate as $W^{1,2}(\|T\|)$ is complete.
It follows (see for instance \cite[Theorem VIII.15]{reed_methods_1972}) that $Q_T$ is the quadratic form of a unique self-adjoint operator on $L^2(\|T\|)$.
\end{proof}

 \begin{definition}
 By the min-max values $\mu_k(A)$ of an unbounded nonnegative self-adjoint operator $A$ on a Hilbert space $H$ with domain $\mathcal{D}(A)$, we mean
 \begin{equation}
 \mu_k(A) := \inf_{\substack{\{\phi_1, \dots, \phi_k \}\subset H\\ H-\text{orthonormal} } } (\phi , A \phi )_H. 
 \end{equation}
 \end{definition}
 
From standard functional analysis (cf.~e.g.~\cite[XIII.1]{reed_methods_1978}), it follows that there are two options. It could be that $\mu_k \to \infty$ as $k \to \infty$. 
In that case, the spectrum of $\Delta_T$ completely consists of eigenvalues $\mu_k$. 
Alternatively, there is a $\mu \geq 0$, the bottom of the essential spectrum of $A$, that is
\begin{equation}
\mu := \inf\{ \lambda \, | \, \lambda \in  \sigma_\text{ess}(A) \},
\end{equation}
and if $\mu_k<\mu$, it is the $k$th eigenvalue counting degenerate eigenvalues a number of times equal to their multiplicity.
If $\mu_k = \mu$, then also $\mu_j = \mu$ for all $j > k$. 
We note that this second case can occur in our setting, and that $-\Delta_T$ does not necessarily have a compact inverse.
 
 For $f \in W^{1,2}(\|T\|)$, it holds that $\mathcal{E}_T(f) = Q_T(f,f)$, therefore the operator $\Delta_T$ defined in Theorem \ref{th:DefLapl} satisfies
 \begin{equation}
 \begin{split}
 \mu_k(-\Delta_T) &= \inf_{\substack{\{\phi_1 , \dots , \phi_{k}\} \subset \mathcal{V} \\ L^2(\|T\|) - \text{orthonormal} }} \sup_{i=1,\dots,k}
 	 Q_T(\phi_i,\phi_i)\\
 &= \inf_{\substack{\{\phi_1 , \dots , \phi_{k}\} \subset \mathcal{V} \\ L^2(\|T\|) - \text{orthonormal} }} \sup_{i=1,\dots,k}
 	 \mathcal{E}_T(\phi_i)\\
 &= \lambda_k(T).
 \end{split}
 \end{equation}
 
 Hence, when we combine Theorem \ref{th:DefLapl} and Theorem \ref{Th:SemContFirst}, we obtain semicontinuity of min-max values of the self-adjoint operators under intrinsic flat convergence without loss of volume.
 
 \begin{theorem}
 \label{th:SemContHilb}
 Let $M_i = (X_i,d_i,T_i)$, be a sequence of integral current spaces converging in the intrinsic flat distance to a nonzero integral current space $M_\infty = (X_\infty,d_\infty,T_\infty)$ such that also $\mathbf{M}(T_i) \to \mathbf{M}(T_\infty)$ as $i \to \infty$. 
 Moreover, assume that the norm on $\|T_i\|$-a.e. approximate tangent space to $\set(T_i)$ is Hilbert, $i=1,2,\dots, \infty$. 
 Then, for every $i$, we may define an unbounded operator $\Delta_{T_i}$ on $L^2(\|T_i\|)$ as in Theorem \ref{th:DefLapl}, and its min-max values satisfy 
 \begin{equation}
 \limsup_{i \to \infty} \mu_k(-\Delta_{T_i}) \leq \mu_k(-\Delta_{T_\infty}).
 \end{equation}
 \end{theorem}
 
 Theorem \ref{th:main} in the introduction is a translation of Theorem \ref{th:SemContHilb} to the simpler setting of closed oriented Riemannian manifolds.
 
 When we combine Theorem \ref{th:DefLapl} together with Theorem \ref{th:SemContMinMaxFlat}, we obtain a version of a Theorem proven by the author in \cite{portegies_semicontinuity_2012} with simpler methods.
 
 \begin{corollary}
 Suppose $T_i$ ($i=1,2,\dots$) are integral currents on Euclidean space, converging in the flat sense to a limit current $T_\infty$ such that additionally $\mathbf{M}(T_i) \to \mathbf{M}(T_\infty)$. 
 Then the min-max values of the unbounded operators $\Delta_{T_i}$, $i=1,2, \dots, \infty$, defined as in Theorem \ref{th:DefLapl}, satisfy
 \begin{equation}
 \limsup_{i \to \infty} \mu_k(-\Delta_{T_i}) \leq \mu_k(-\Delta_{T_\infty}).
 \end{equation}
 \end{corollary}
 
 \begin{proof}
 As the currents $T_i$ are defined on Euclidean space, the approximate tangent spaces automatically have a Hilbert structure. The rest follows from Theorems \ref{th:DefLapl} and \ref{th:SemContMinMaxFlat}.
 \end{proof}
 
\section{Examples}
\label{se:example}

\subsection{Example: a disappearing spline}
We first consider the example of a disappearing spline. This shows that when a sequence of manifolds converges in the intrinsic flat sense without loss of volume, that is when it satisfies the conditions of Theorem \ref{th:main}, we may not expect \emph{continuity} of the eigenvalues. 
Consider for $\epsilon > 0$ the smooth functions $h_\epsilon:[-2,2] \to \R$, $h_\epsilon\geq 0$, such that
\begin{equation}
h_\epsilon(x) = \begin{cases}
\sqrt{1-(x+1)^2} & -2 \leq x \leq -\epsilon, \\
\epsilon &  \epsilon \leq x \leq 2 - \epsilon, \\
\sqrt{ \epsilon^2 - (x - 2 + \epsilon)^2 } &  2 - \epsilon \leq x \leq 2.
\end{cases}
\end{equation}
Moreover, assume that $h_\epsilon$ is decreasing on $(-1,2)$.
We construct manifolds $M_\epsilon$ by revolving the graphs $y = h_\epsilon(x)$ around the $x$-axis.
Sormani has shown that in this case the manifolds $M_\epsilon$ converge in the intrinsic flat sense as $\epsilon \downarrow 0$ to the unit sphere $S^1$ \cite[Example A.4]{sormani_intrinsic_2011}.
Also observe that $\mathrm{Vol}(M_\epsilon) \to \mathrm{Vol}(S^1)$.
We choose test functions $f_\epsilon$ on $M_\epsilon$, that only depend on $x$, increasing in $x$, such that 
\begin{equation}
f_\epsilon(x) := \begin{cases}
- c_\epsilon & x \leq -\epsilon,\\
\sin\left(\frac{\pi x}{4} \right) & \epsilon \leq x \leq 2 - \epsilon.
\end{cases}
\end{equation}
where $c_\epsilon$ is chosen in such a way that $f_\epsilon$ has zero average on $M_\epsilon$. 
After calculating the Rayleigh-quotien for $f_\epsilon$, we observe that
\begin{equation}
\limsup_{\epsilon \downarrow 0} \lambda_1(M_\epsilon)
\leq \limsup_{\epsilon\downarrow 0} \frac{\int_{M_\epsilon} |\nabla f_\epsilon|^2 d \Ha^2 }{\int_{M_\epsilon} |f_\epsilon |^2 d\Ha^2} 
= \left(\frac{\pi}{4}\right)^2 < 2 = \lambda_1( S^1 ).
\end{equation}
This shows that the first eigenvalue can actually jump up in the limit.

\subsection{Example: cancellation can make eigenvalues drop}

We conclude with an example that shows that we cannot remove the assumption of the convergence of the volumes from Theorem \ref{th:main}. 
Without this assumption, cancellation can occur, which can make the eigenvalues drop down in the limit.

For $x \in \R^3$ let $Q_{q,r}(x)\subset \R^3$ denote the following rectangular box:
\begin{equation}
Q_{q,r}(x) = [x_1 - q , x_1 + q]
\times [x_2 - q, x_2 + q]
\times [x_3 - r, x_3 + r].
\end{equation}
Let $e_1$ denote the unit vector $(1,0,0)\in \R^3$.

Consider the sequence $M_j = \partial W_j$ where $W_j$ is given as a set by
\begin{equation}
W_j = Q_{1,1}(2 e_1) \cup Q_{1,1}(- 2 e_1) \cup Q_{1,4^{-j}}(0).
\end{equation}
At first sight, one might be tempted to think that as $j\to\infty$ the $M_j$ converge in the intrinsic flat sense to the boundary of the two cubes $Q_{1,1}(2 e_1)$ and $Q_{1,1}(-2 e_1)$. However, the induced embedding of $M_j$ into $\mathbb{R}^3$ is not \emph{isometric}. Yet, we can use the construction by Sormani in \cite[Example A.19]{sormani_intrinsic_2011}, and create manifolds $\tilde{M}_j$ with a lot of tunnels from one side of the thin sheet to the other side. 
To be more precise, let
\begin{equation}
P_j := Q_{1,4^{-j}}(0) \backslash 
	\bigcup_{k,\ell = -2^j+1}^{2^j - 1} Q_{4^{-j},4^{-j}}
		\left(\frac{k}{2^{j}},\frac{\ell}{2^{j}},0\right)
\end{equation}
and
\begin{equation}
\tilde{W_j} = Q_{1,1}(2 e_1) \cup P_j \cup Q_{1,1}(- 2 e_1).
\end{equation}
and let $\tilde{M}_j = \partial \tilde{W}_j$ (to be interpreted as the boundary of the set in Euclidean space).
We claim that 
\begin{equation}
\tilde{M}_j \to \partial\left( Q_{1,1}(2e_1) \cup Q_{1,1}(-2e_1) \right) =: \tilde{M}.
\end{equation}
in the intrinsic flat sense, where $\tilde{M}$ is endowed with the distance $d_Y$, which is the induced length distance on the space $Y$ given by
\begin{equation}
Y = \partial Q_{1,1}(2 e_1) \cup \partial Q_{0,1}(0) \cup \partial Q_{1,1}(-2e_1),
\end{equation}
which is the Gromov-Hausdorff limit of $\tilde{M}_j$.

By the isometric product $A\times B$ of two metric spaces $(A,d_A)$ and $(B,d_B)$ we mean the Cartesian product endowed with the distance
\begin{equation}
d_{A\times B} ((a_1,b_1),(a_2,b_2)) = \sqrt{(a_1-a_2)^2 + (b_1-b_2)^2}.
\end{equation}
We consider the metric space
\begin{equation}
\begin{split}
Z_j &:= \left( \tilde{M}_j \times [0,\tfrac{1}{j}]\right) \cup  
\left( Y \times [-\tfrac{1}{j},0]\right)\\
& \qquad \cup \left(\left(\partial Q_{1,1}(2 e_1) \cup Q_{1,4^{-j}}(0) \cup \partial Q_{1,1}(-2e_1)\right) \times \{0\} \right).
\end{split}
\end{equation}
Note that for $j$ large enough, the embeddings 
\begin{alignat}{2}
\phi_j &: M_j \to Z_j ,\quad & \phi_j(x) &= (x,\tfrac{1}{j}),\\
\psi_j &: Y \to Z_j  ,\quad & \psi_j(y) &= (y, -\tfrac{1}{j}),
\end{alignat}
are isometric. 
Let $B_j \in I_3(Z_j)$ be the current
\begin{equation}
B_j := \llbracket \tilde{M}_j \times [0,\tfrac{1}{j}] \rrbracket 
	  + \llbracket \tilde{M} \times [-\tfrac{1}{j} ,0] \rrbracket
	  - \llbracket P_j \times \{0\} \rrbracket.
\end{equation}
Then 
\begin{equation}
(\phi_j)_\# \llbracket \tilde{M}_j \rrbracket - (\psi_j)_\# \llbracket \tilde{M} \rrbracket = \partial \llbracket B_j \rrbracket.
\end{equation}
Since
\begin{equation}
\Mass(\llbracket B_j \rrbracket)  \leq \frac{1}{j} \left( \mathrm{Vol}(\tilde{M}_j) + \mathrm{Vol}(\tilde{M}) \right) + \mathrm{Vol}(Q_{1,4^{-j}}(0)) \to 0,
\end{equation}
as $j \to \infty$, it follows that indeed $\tilde{M}_j$ converges to $\tilde{M}$ in the intrinsic flat sense.

Recall that for a smooth $n$-dimensional manifold $M$, the Cheeger's constant $h(M)$ is defined as
\begin{equation}
h(M) := \inf_E \frac{\Ha^{n-1}(E)}{\min(\Ha^{n}(A),\Ha^n(B))},
\end{equation}
where the infimum runs over all smooth $(n-1)$-dimensional submanifolds $E$ of $M$ that divide $M$ into two disjoint submanifolds $A$ and $B$ \cite{cheeger_lower_1970}. 
Cheeger's inequality states that
\begin{equation}
\lambda_1(M) \geq \frac{h(M)^2}{4}.
\end{equation}

As we can take a minimizing function with average zero that is constant on each of the two cubes, we have $\lambda_1 ( \tilde{ M } ) = 0$. 
However, Cheeger's inequality implies that $\lambda_1(\tilde{M}_j)$ is uniformly bounded away from zero. The idea is as follows. Let $E$ be a $1$-dimensional submanifold of $\tilde{M}_j$, separating $\tilde{M}_j$ into two disjoint submanifolds $A$ and $B$, such that $\Ha^{2}(A) \leq \Ha^2(B)$.
Without loss of generality we may assume that $A$ is connected.
If $\dia(E) \leq 1$, there are constants $c_1$ and $c_2$ (independent of $j$) such that
\begin{equation}
\Ha^2(A) \leq c_1 \dia(E)^2, \qquad \Ha^1(E) \geq c_2 \dia(E),
\end{equation}
so that
\begin{equation}
\frac{\Ha^1(E)}{\Ha^2(A)} \geq \frac{c_2}{c_1} \frac{1}{\dia(E)} \geq \frac{c_2}{c_1}.
\end{equation}
If $\dia(E) > 1$, $\Ha^1(E) > c_3 > 0$ and 
\begin{equation}
\frac{\Ha^1(E)}{\Ha^2(A)} \geq \frac{c_3}{2 \Ha^2(\tilde{M_j})}.
\end{equation}
As the volume $\Ha^2(\tilde{M_j})$ is uniformly bounded, the uniform lower bound on $h(\tilde{M}_j)$ follows.

It is true that Cheeger's inequality in this form applies to smooth manifolds, but we can for instance approximate the spaces in the flat distance in Euclidean space to conclude the bound still holds.
Therefore, in the case of this example,
\begin{equation}
\limsup_{j \to \infty} \lambda_1(\tilde{M}_j) > \lambda_1(\tilde{M}).
\end{equation}
The example can be easily modified to obtain a connected limit space, by adding a thin tube connecting the one cube to the other. In that case $\lambda_1$ of the limit space can be made arbitrarily small by making the tube arbitrarily thin.

\appendix

\section{Decomposition of one-dimensional integral currents into curves}
\label{se:Decomposition}
In this section we will prove that any one-dimensional integral current can be decomposed in the sum of pushforwards under Lipschitz functions of currents associated to intervals. 
The Euclidean result is well-known, and can be found in \cite[4.2.25]{federer_geometric_1996}. 
The proof relies on the deformation theorem, however, which is not available in arbitrary metric spaces. 
To obtain the decomposition result, we therefore approximate the original currents by finite-dimensional ones, and take an appropriate limit.

The result is an easier version of the decomposition theorems of normal one-dimensional currents by Paolini and Stepanov \cite{paolini_decomposition_2012,paolini_structure_2013}. 
The decomposition theorem for integral currents can also be proved using their results as a starting point, but we chose to instead give a simpler argument, which is possible because the currents are integral. However, there are many parallels with the proofs in \cite{paolini_decomposition_2012}.

Throughout this section, $X$ will denote a complete metric space.

We first define what it means for an integral current (of arbitrary dimension) to be (in)decomposable. 
This definition is exactly the same as in the Euclidean case.

\begin{definition}
A current $T \in I_n(X)$ is called decomposable if there are two nonzero currents $T_1, T_2 \in I_n(X)$ such that
\begin{equation}
T = T_1 + T_2, \qquad \Mass(T) = \Mass(T_1) + \Mass(T_2), \qquad \Mass(\partial T) = \Mass(\partial T_1) + \Mass(\partial T_2).
\end{equation}
If a current is not decomposable, it is called indecomposable.
\end{definition}

The following theorem is the main result of this section. 
It is an immediate consequence of Proposition \ref{pr:Decompn} and Lemmas \ref{le:acyclicDecomp} and \ref{le:cyclicDecomp} that follow below.

\begin{theorem}
\label{th:DecompositionOneD}
Let $T \in I_1(X)$.
Then, there is a sequence $T_1, T_2, \dots$ in $I_1(X)$ such that
\begin{equation}
T = \sum_{i=1}^\infty T_i, \qquad \Mass(T) = \sum_{i=1}^\infty \Mass(T_i), \qquad \Mass(\partial T) = \sum_{i=1}^\infty \Mass(\partial T_i),
\end{equation}
and for every $i=1, 2, \dots$ there is a Lipschitz function $\sigma_i:[0,L_i] \to X$ with $\Lip(\sigma_i) \leq 1$, such that 
\begin{equation}
(\sigma_i)_\# \llbracket \chi_{[0,L_i]} \rrbracket = T_i.
\end{equation}
\end{theorem}

Just as in the Euclidean case, integral $n$-currents can be decomposed into indecomposable pieces.

\begin{proposition}
\label{pr:Decompn}
Let $T \in I_n(X)$.
Then, there is a sequence of indecomposable integral $n$-currents $T_1, T_2, \dots$ such that
\begin{equation}
T = \sum_{i=1}^\infty T_i, \qquad \Mass(T) = \sum_{i=1}^\infty \Mass(T_i), \qquad \Mass(\partial T) = \sum_{i=1}^\infty \Mass(\partial T_i).
\end{equation}
\end{proposition}
\begin{proof}
We may assume without loss of generality that $X = \ell^\infty$ (indeed, we may first restrict $X$ to $\overline{\set(T)} \subset X$, which is separable, and then use the Kuratowski embedding to embed $\overline{\set(T)}$ into $\ell^\infty$).
The proof of the proposition is then exactly the same as in the Euclidean case (cf. \cite[4.2.25]{federer_geometric_1996}),
except one needs to replace the Euclidean isoperimetric inequality by the isoperimetric inequality in $\ell^\infty$, that is, use the fact that there is a constant $C(n)$ such that for any $S \in I_n(\ell^\infty)$ with $\partial S = 0$ there exists an $S_0 \in I_{n+1}(\ell^\infty)$ with $\partial S_0 = S$ and
\begin{equation}
\Mass(S_0) \leq C(n) \Mass(S)^{\frac{n+1}{n}}.
\end{equation}
This fact is proved by Ambrosio and Kirchheim in \cite[Appendix B]{ambrosio_currents_2000}.
\end{proof}

In order to obtain a decomposition, we approximate the current $T$ by finite-dimensional projections, use the standard decomposition theorem for those, and take a weak limit to obtain a decomposition for the original current. 
The decomposition is non-unique, and one needs to be careful to select which components of the approximate currents converge to a component of the original current. 
We initially only keep track of components with boundary.

\begin{lemma}
\label{le:acyclicDecomp}
Let $T \in I_1(X)$.
Then there exist one-dimensional integral currents $T_i$, for $i = 0, 1, \dots, N = \Mass(\partial T) / 2$, such that $T = \sum_{i=0}^N T_i$, $\partial T_0 = 0$, and 
\begin{equation}
\Mass(T) = \sum_{i=0}^N \Mass(T_i), \qquad \Mass(\partial T) = \sum_{i=0}^N \Mass(\partial T_i),
\end{equation}
such that for $i=1, \dots, N$, 
\begin{equation}
T_i = (\sigma_i)_\# \llbracket  \chi_{[0,L_i]} \rrbracket,
\end{equation}
for Lipschitz functions $\sigma_i : [0,L_i] \to X$ with $\Lip(\sigma_i) \leq 1$.
\end{lemma}

\begin{proof}
Again, by using the Kuratowski embedding, we may assume that $X = \ell^\infty$. 
Let $K_j$ be a sequence of compact sets
$K_j \subset K_{j+1}$, $\|T\|(X \backslash K_j) \leq 1/j$, 
such that $\|T\|$ is concentrated on $\cup_j K_j$ and $\|\partial T\|$ is concentrated on $\cap_j K_j$. 
We  may also assume that $T$ is supported on $B_R(0)$ for some $R > 0$.  
Indeed, since $T$ has finite mass, there exists an $R > 0$ such that 
\begin{equation}
\partial (T \llcorner (X \backslash B_R(0)) ) = - \langle T, \|. \|_\infty, R \rangle = 0,
\end{equation}
so that $T\llcorner (X\backslash B_R(0))$ can be included in $T_0$.

Since $\ell^\infty$ has the metric approximation property, there are projections $P^j: \ell^\infty \to \ell^\infty$ such that the operator norm of $P^j$ is less than $1$, and for every $x \in K_j$, $\|P^j x - x\|_\infty \leq 1/j$. 
Paolini and Stepanov recorded a proof of this statement in \cite[Lemma 5.7]{paolini_decomposition_2012}.

It holds that $P^j_\# T \rightharpoonup T$ weakly.
Indeed, if $\omega = (f, \pi_1, \dots, \pi_n) \in \mathcal{D}^{n+1}(X)$, with $\Lip(\pi_i) \leq 1$ for all $i=1,\dots,n$, then
\begin{equation}
\begin{split}
|P^j_\# T (\omega) - T(\omega)|
&= |T( (P^j)^\#\omega ) - T(\omega) |\\
&\leq |T ( f \circ P^j, \pi_1 \circ P^j, \dots, \pi_n \circ P^j ) - T(f\circ P^j, \pi_1, \dots, \pi_n) | \\
&\quad + |T(f \circ P^j, \pi_1, \dots, \pi_n) - T(f, \pi_1, \dots, \pi_n) |.
\end{split}
\end{equation}
By \cite[Theorem 5.1]{ambrosio_currents_2000},
\begin{equation}
\begin{split}
&|T ( f \circ P^j, \pi_1 \circ P^j, \dots, \pi_n \circ P^j ) - T(f\circ P^j, \pi_1, \dots, \pi_n) | \\
&\qquad \leq \sum_{i=1}^n \Big( \int_X |f\circ P^j| |\pi_i - \pi_i \circ P^j| d\| \partial T \| \\
&\qquad \qquad + \Lip(f \circ P^j) \int_{\supp f} |\pi_i - \pi_i\circ P^j| d\|T\|\Big) \\
&\qquad \leq \frac{n}{j} \int_X |f \circ P^j| d\|\partial T\| + \frac{n}{j} \Lip(f) \Mass(T) + 2 n R \|T\|\left(X \backslash \bigcup_{\ell=1}^j K_\ell \right),
\end{split}
\end{equation}
which tends to $0$ as $j\to\infty$. 
The other term is estimated by
\begin{equation}
\begin{split}
&|T (f \circ P^j, \pi_1, \dots, \pi_n) - T(f , \pi_1, \dots, \pi_n )|  \\
&\qquad \leq \int_X |f \circ P^j - f| d\|T\| \\
&\qquad \leq \Lip(f) \frac{1}{j} \Mass(T) + 2 \sup_{B_R(0)} (f) \|T\|\left(X \backslash \bigcup_{\ell=1}^j K_\ell \right),
\end{split}
\end{equation}
which also tends to $0$ as $j\to \infty$.
This showed that $P^j_\#T \rightharpoonup T$ weakly. 

Since $P^j$ is contractive, $\Mass(P^j_\# T) \leq \Mass(T)$ and since the mass is lower semicontinuous under weak convergence, $\Mass(P^j_\# T) \to \Mass(T)$.
It follows also that $\| P^j_\# T \| \rightharpoonup \| T\|$ in the weak sense of measures.
 
By the corresponding theorem in the Euclidean case, which can be proved using the deformation theorem, we know that $P^j_\# T$ can be decomposed, that is, for $i$ large enough, there exist (indecomposable for $i \geq 1$) integral currents $T_i^j$ such that $P^j_\# T = \sum_{i=0}^N T_i^j$, $\partial T_0^j = 0$, and  
\begin{equation}
\Mass(P^j_\# T) = \sum_{i=0}^N \Mass(T_i^j), \qquad \Mass(\partial (P^j_\# T) ) = \sum_{i=0}^N \Mass(\partial T_i^j).
\end{equation} 
Moreover, we note that the measures $\|T_i^j\|$ are uniformly tight, as they are absolutely continuous with respect to $\|P_\#^j T\|$, and these measures are uniformly tight as they weakly converge to $T$ (see also \cite[Lemma B.2]{paolini_decomposition_2012}). Therefore, by the compactness theorem \cite[Theorem 5.2]{ambrosio_currents_2000} and closure theorem \cite[Theorem 8.5]{ambrosio_currents_2000}, possibly selecting a subsequence, we can assume that $T_i^j \rightharpoonup T_i$ as $j \to \infty$ weakly as currents.
Set $L_i := \Mass(T_i)$.

Again, by the Euclidean decomposition theorem, we know that there exist Lipschitz functions $\sigma_i^j:[0,L_i]\to X$ such that $\Lip(\sigma_i^j) \leq \Mass(T_i^j) / L_i$, and 
\begin{equation}
(\sigma_i^j)_\# \llbracket \chi_{[0,L_i]} \rrbracket = T_i^j.
\end{equation}
By a slight variation of the Arzela-Ascoli Theorem, for a subsequence, $\sigma_i^j \to \sigma_i$ strongly in the continuous topology, where $\Lip(\sigma_i) \leq 1$. 
(One way to see this, is to select a dense sequence of points $p_\ell \in \set(T_i)$, and to note that for every $\ell$, since $\|T_i^j\| \to \|T_i\|$ weakly in the sense of measures, there are values $t_i^j$ such that $\sigma_i^j(t_i^j) \to p_\ell$.)
Therefore, by the basic continuity property of currents, $(\sigma_i^j)_\# \llbracket \chi_{[ 0,L_i^j]} \rrbracket \rightharpoonup (\sigma_i)_\# \llbracket \chi_{[0,L]} \rrbracket$ as $j \to \infty$, and thus by uniqueness of weak limits,
\begin{equation}
(\sigma_i)_\# \llbracket \chi_{[0,L_i]} \rrbracket = T_i.
\end{equation}
\end{proof}

We are left with finding the components that have no boundary.
However, we can easily reduce to the previous case in case of indecomposable cyclic currents. 
This is the content of the following lemma.

\begin{lemma}
\label{le:cyclicDecomp}
Let $T \in I_1(X)$ be indecomposable, with $\partial T = 0$. 
Then, with $L= \Mass(T)$, there exists a curve $\sigma : [0,L] \to X$, $\Lip(\sigma) \leq 1$, such that
\begin{equation}
\sigma_\# \llbracket \chi_{[0,L]} \rrbracket = T.
\end{equation}
\end{lemma}

\begin{proof}
Take a point $p \in \set T$.
Let $\rho: Y \to \R$ be given by $\rho(x) = d(p,x)$. 
Then, for almost every $r$, $T \llcorner B_r(p)$ is an integral current.
For ($\lebmeas^1$-a.e.) small enough $r$, $\partial(T \llcorner B_r(p))$ is nonzero, otherwise $T$ would be decomposable.
For such $r$, we apply Lemma \ref{le:acyclicDecomp} to $T \llcorner B_r(p)$ and $T \llcorner (X \backslash B_r(p))$, to construct $\sigma$.
\end{proof}

\section{Many one-dimensional slices have good connected components}

In this section, we will slice integral currents according to an extension of $n-1$ coordinate functions in charts, to obtain one-dimensional slices.
The main conclusion is, that $\Ha^n$-a.e., locally for a large proportion of such slices, and for a large proportion of pairs of points in the same slice, there exists a component of the slice that connects the two points.

In this section and the next, for $x \in \R^n$ and $r > 0$, $Q_r(x) \subset \R^n$ will stand for the cube centered at $x$ with edge length $2r$.

\begin{lemma}
\label{le:GoodOneDSlices}
Let $Y$ be a $w^*$-separable Banach space and let $T\in I_n(Y)$. 
Let $T$ have the parametrization
\begin{equation}
T = \sum_{\ell=1}^\infty \vartheta_\ell (g_\ell)_\# \llbracket \chi_{K_\ell}  \rrbracket,
\end{equation} 
with $\vartheta_\ell \in \N$, compact sets $K_\ell \subset \R^n$ and Lipschitz maps $g_\ell: \R^n \to Y$, with the properties described in Lemma \ref{le:GoodParametrization}.

Fix an $\ell \in \mathbb{N}$. 
Let $\hat{P}_i:\R^n \to \R^{n-1}$, for any $i=1, \dots, n$, be the orthogonal projection given by
\begin{equation}
\hat{P}_i(x_1, \dots, x_n) := (x_1, \dots, x_{i-1}, x_{i+1} , \dots, x_n),
\end{equation}
and let $W: Y \to \R^{n}$ be a Lipschitz function with $\Lip(W_i) \leq 2 \sqrt{n}$ such that for all $x \in K_\ell$, $x = W(g_\ell(x))$, and define $\hat{W}_i = \hat{P}_i \circ W$. (Such a $W$ exists by (\ref{eq:ComparisonToEuclidean}).)

Given $0 < \delta$, for $\lebmeas^n$-a.e. $x \in K_\ell$, there exists an $r_0$ such that for $\lebmeas^1$-a.e. $0 < r \leq r_0$, 
there exists a compact set $K^r \subset Q_r(x) \cap K_\ell$, with $\Ha^n(K^r) \geq (1 - \delta) (2r)^n$ such that for every $i=1,\dots,n$, and $x_1, x_2 \in K^r$ with $\hat{P}_i(x_1) = \hat{P}_i(x_2)$, with $t_1:= P_i(x_1) \leq P_i(x_2) =:t_2$, 
\begin{equation}
\langle T, \hat{W}_i, \hat{P}_i(x_1) \rangle \llcorner \{ t_1 < W_i \leq t_2\},
\end{equation}
is a one-dimensional integral current concentrated on 
\begin{equation}
\bigcup_\ell g_\ell(G_\ell) \cap \hat{W_i}^{-1}(\hat{P}_i(x_1)) \cap W_i^{-1}([t_1,t_2]), 
\end{equation}
and for all its decompositions (in the sense of Theorem \ref{th:DecompositionOneD}) there are exactly $\vartheta_\ell$ components $\sigma_j:[0,L_j] \to Y$, $j = 1, \dots, \vartheta_\ell$, such that
\begin{equation}
\begin{cases}
\sigma_j(0) = g_\ell( x_1 ), \quad \sigma_j(L_j) = g_\ell(x_2),& \text{if } n-i \text{ even},\\
\sigma_j(0) = g_\ell( x_2 ), \quad \sigma_j(L_j) = g_\ell(x_1),& \text{if } n-i \text{ odd}.
\end{cases}
\end{equation} 
Moreover, for these $\sigma_j$ it holds that $\sigma_j([0,L_j]) \subset B_R(g_1(x))$, where $R := (3 + \sqrt{n})r$.
\end{lemma}

\begin{proof}
It suffices to prove the lemma for $i = n$ and $\ell = 1$. 

By the Ambrosio-Kirchheim slicing theorem \cite[Theorem 5.6]{ambrosio_currents_2000}, for $x \in K_1$,
\begin{align}
\int_{\R^{n-1}} \| \langle \partial T, \hat{W}_n, t \rangle \| (B_R(g_1(x))) dt &\leq (2 \sqrt{n})^{n-1}\|\partial T\|(B_R(g_1(x))), \\
\int_{\R^{n-1}} \| \langle T, \hat{W}_n, t \rangle \| (B_R(g_1(x)) \backslash g_1(K_1)) dt & \leq (2 \sqrt{n})^{n-1}\| T \|(B_R(g_1(x)) \backslash g_1(K_1)), \\
\int_{\R^n} \| \langle T, W, t \rangle \| (B_R(g_1(x)) \backslash g_1(K_1) ) dt 
&\leq (2 \sqrt{n})^{n} \|T \|(B_R(g_1(x)) \backslash g_1(K_1)).
\end{align}
Hence, since for $\lebmeas^{n-1}$-a.e. $t \in \R^{n-1}$, $\langle \partial T, \hat{W}_n, t \rangle = (-1)^n \partial \langle T, \hat{W}_n, t \rangle$,
\begin{equation}
\lebmeas^{n-1}\{ t \in \R^{n-1} \, | \, \| \partial \langle T, \hat{W}_n, t \rangle \| (B_R(g_1(x))) \geq 1 \}
\leq (2 \sqrt{n})^{n-1} \|\partial T\|(B_R(g_1(x))),
\end{equation}
and,
\begin{align*}
&\lebmeas^{n-1}\{ t \in \R^{n-1} \, | \, \| \langle T, \hat{W}_n, t \rangle \| (B_R(g_1(x)) \backslash g_1(K_1)) \geq \frac{R}{3 + \sqrt{n}} \} \\
&\qquad \leq  (2 \sqrt{n})^{n-1} \frac{3 + \sqrt{n}}{R} \| T \|(B_R(g_1(x)) \backslash g_1(K_1)), \\
&\lebmeas^n \{ y \in \R^n \, | \, \|\langle T, W, y \rangle \|(B_R(g_1(x)) \backslash g_1(K_1)) \geq 1\}\\
&\qquad \leq (2 \sqrt{n})^{n} \| T \|(B_R(g_1(x)) \backslash g_1(K_1)).
\end{align*}

The ($n-1$)-dimensional upper density $\Theta_{n-1}^*(\|\partial T\|, p)$ vanishes for $\Ha^{n-1}$-a.e. $p \in \set T \backslash \set \partial T$, and in particular, $\Theta_n^*(\|\partial T\|, g_1(x))$ vanishes for $\lebmeas^n$-a.e. $x \in K_1$. 
Moreover, $\Theta_n^*(\|T\| \llcorner (Y \backslash g_1(K_1)), g_1(x))= 0$ for $\lebmeas^n$-a.e. $x \in K_1$. 

It follows from the above and from \cite[Theorems 5.6, 5.7, 5.8]{ambrosio_currents_2000} that for $\lebmeas^n$-a.e. $x \in K_1$, there exists an $r_0$ such that for $\lebmeas^1$-a.e. $r\leq r_0$, there exist compact sets $A_r \subset Q_r(\hat{P}_n(x))$ and $K^r \subset Q_r(x)$ such that 
\begin{itemize}
\item $\hat{P}_n^{-1}(A_r) \cap K^r = K^r$,
\item $\Ha^n(K^r) \geq (1- \delta) (2r)^n$,
\item for every $a \in A_r$, $\langle T, \hat{W}_n, a \rangle$ is in $I_1(Y)$ and is concentrated on 
\begin{equation}
\bigcup_\ell g_\ell(K_\ell) \cap \hat{W_i}^{-1}(a) , 
\end{equation}
with $\partial \langle T, \hat{W}_n, a \rangle \llcorner B_R(g_1(x)) = 0$ and  
\begin{equation}
\label{eq:NotMuchOutsideOneDSlice}
\|\langle T, \hat{W}_n, a \rangle \| (B_R(g_1(x)) \backslash g_1(K_1) ) < r,
\end{equation}
\item for every $y \in K^r$, $\langle T, W, y \rangle$ is in $I_0(Y)$, concentrated on 
\begin{equation}
\bigcup_{\ell} g_k(K_\ell) \cap W^{-1}(y),
\end{equation}
with
\begin{equation}
\begin{split}
\langle T, W, y \rangle 
&= \langle \langle T, \hat{W}_n,\hat{P}_n(y) \rangle, W_n, P_n(y) \rangle \\
&= (\partial \langle T, \hat{W}_n, \hat{P}_n(y) ) \llcorner \{ P_n(y) < W_n \}  \\
&\quad - \partial ( \langle T, \hat{W}_n, \hat{P}_n(y) \llcorner \{ P_n(y) < W_n \} ),
\end{split}
\end{equation}
and
\begin{align}
\| \langle T, W, y \rangle \| (B_R(g_1(x)) \backslash g_1(K_1)) &= 0, \\
\langle T, W, y \rangle \llcorner B_R(g_1(x)) &= \vartheta_1 \delta_{g_1(y)}.
\end{align}
\end{itemize}
Let $a \in A_r$ and $t_1 < t_2$ such that $(a, t_i) \in K^r$, $i = 1, 2$. 
Write for ease of notation
\begin{equation}
T_{t_1, t_2} := \langle T, \hat{W}_n, a \rangle \llcorner \{ t_1 < W_n \leq t_2 \}.
\end{equation}
Note that 
\begin{equation}
\begin{split}
\partial  T_{t_1,t_2}
&= (\partial \langle T, \hat{W}_n, a \rangle ) \llcorner \{ t_1 < W_n \} 
- \langle \langle T, \hat{W}_n, a \rangle, W_n, t_1 \rangle \\
&\quad - (\partial \langle T, \hat{W}_n, a \rangle) \llcorner \{ t_2 < W_n \} 
+ \langle \langle T, \hat{W}_n, a \rangle, W_n, t_2 \rangle,
\end{split}
\end{equation}
so that, since $\partial \langle T, \hat{W}_n, a \rangle \llcorner B_R(g_1(x)) = 0$,
\begin{equation}
\begin{split}
\partial T_{t_1,t_2} \llcorner B_R(g_1(x))
&= - \langle \langle T, \hat{W}_n, a \rangle, W_n, t_1 \rangle \llcorner B_R(g_1(x)) 
	+ \langle \langle T, \hat{W}_n, a \rangle, W_n, t_2 \rangle \llcorner B_R(g_1(x)) \\
&= - \langle T, W, (a,t_1) \rangle \llcorner B_R(g_1(x))
	+ \langle T, W, (a,t_2) \rangle\llcorner B_R(g_1(x)) \\
&= \vartheta_1 \delta_{g_1(x_2)} - \vartheta_1 \delta_{g_1(x_1)}.
\end{split}
\end{equation}

Decomposing $T_{t_1, t_2}$ as in Theorem $\ref{th:DecompositionOneD}$, we find that there are Lipschitz maps $\sigma_j:[0,L_j]\to Y$, 
with $\Lip(\sigma_j) \leq 1$, 
and $T_j \in I_1(X)$ ($j=1,2, \dots$) given by
\begin{equation}
T_j := (\sigma_j)_\# \llbracket \chi_{[0,L_j]} \rrbracket,
\end{equation}
which are indecomposable and
\begin{equation}
T_{t_1, t_2} = \sum_{j=1}^\infty T_j, 
\qquad \Mass(T_{t_1, t_2}) = \sum_{j=1}^\infty \Mass(T_j), 
\qquad \Mass(\partial T_{t_1, t_2}) = \sum_{j=1}^\infty \Mass(\partial T_j).
\end{equation}
Without loss of generality, we may assume that $\sigma_j (0) = g_1(x_1)$ for all $j = 1, \dots, n$. 
We will now show that for these values of $j$, $\sigma_j([0,L_j]) \subset B_R(g_1(x))$. 
Indeed, because 
\begin{equation}
\partial \langle T, \hat{W}_n, a \rangle \llcorner B_R(g_1(x)) = 0,
\end{equation}
if  $\sigma_j$ touches $\partial B_R(g_1(x))$, necessarily
\begin{equation}
\Mass((\sigma_j)_\# \llbracket \chi_{[0,L_j]} \rrbracket ) 
\geq R - d(g_1(x_1), g_1(x) ) 
\geq R - \| x_1 - x \| \geq R - \sqrt{n} r \geq 3 r,
\end{equation}
where we used (\ref{eq:ComparisonToEuclidean}).
However, again by (\ref{eq:ComparisonToEuclidean}) and the area formula,
\begin{equation}
\| T_{t_1,t_2} \| (g_1(K_1)) \leq t_2 - t_1 \leq 2r,
\end{equation}
and by (\ref{eq:NotMuchOutsideOneDSlice}),
\begin{equation}
\| T_{t_1,t_2} \| (B_R(g_1(x)) \backslash g_1(K_1)) < r.
\end{equation}
Since $\sigma_j([0,L_j]) \subset B_R(g_1(x))$, necessarily $\sigma_j(L_j) = g_1(x_2)$.
\end{proof}

\section{A Poincar\'{e}-like inequality}
\label{se:Poincare}

In Section \ref{se:DefSobolev}, we showed that any function in $W^{1,2}(\|T\|)$ is tangentially differentiable $\|T\|$-a.e.. 
The main ingredient is a Poincar\'{e}-like inequality, that we will prove in this section.

In the Poincar\'{e}-like inequality below, we will use the notation
\begin{equation}
(f)_G := \frac{1}{\Ha^k(G) }\int_G f d\Ha^{k},
\end{equation}
for an integrable function $f$ and a Borel set $G \subset \R^n$, where the appropriate integer $k$ will be apparent from the context (i.e., $k$ such that $0 < \Ha^k (G)< \infty$).

\begin{theorem}
\label{th:PoincareLike}
Let $Y$ be a $w^*$-separable Banach space and let $T \in I_n(Y)$. Let $T$ have the parametrization
\begin{equation}
T = \sum_{\ell=1}^\infty \vartheta_\ell (g_\ell)_\# \llbracket \chi_{K_\ell} \rrbracket,
\end{equation} 
with $\vartheta_\ell \in \N$, compact $K_\ell \subset \R^n$, and Lipschitz maps $g_\ell : \R^n \to Y$, with the properties described in Lemma \ref{le:GoodParametrization}.

Then, for almost every $x \in K_1$, for small enough $r_0>0$, for almost every $r \leq r_0$, there exists a compact set
$G \subset K_1 \cap Q_r(x) $ such that $\Ha^n(G) \geq (1-\delta) (2r)^n$ and for all Lipschitz $f:\supp(T) \to \R$,
\begin{equation}
\frac{1}{\Ha^n(G)}\int_G | (g_1^\# f) - ( g_1^\# f)_G | d \Ha^n 
	\leq \frac{C(n) r} {\|T\|(B_R(g_1(x)))}\int_{B_R(g_1(x))} |df| d\|T\|,
\end{equation}
with $R := (3 + \sqrt{n})r$.
\end{theorem}

\begin{proof}
First, we fix a $\delta > 0$ and apply Lemma \ref{le:GoodOneDSlices} (with $\delta$ replaced by $\delta^{n+1}$), to obtain a set $K^r$ with the properties mentioned in the lemma, in particular $\Ha^n(K^r) > (1-\delta^{n+1}) (2r)^n$. 

Let $P$ and $W$ be defined as in Lemma \ref{le:GoodOneDSlices}. 
We additionally introduce the notations 
\begin{equation}
P^k = (P_1, \dots, P_k), \qquad \underline{P}^k = (P_{k+1}, \dots, P_n),
\end{equation}
and will use the corresponding definitions for $W^k$ and $\underline{W}^k$.
By the Ambrosio-Kirccheim slicing theorems \cite[Theorems 5.6, 5.7, 5.8]{ambrosio_currents_2000} and the regularity of the Lebesgue-measure, we may without loss of generality assume that for \emph{all} $x \in K_r$, and every $k=1, \dots, n$, $\langle T, W^{k}, P^{k}(x) \rangle \in I_{n - k}(Y)$, while for $k = 2, \dots, n$,
\begin{itemize}
\item for $\lebmeas^1$-a.e. $t \in \R$,
\begin{equation}
\langle \langle T, W^{k-1}, P^{k-1}(x) \rangle, W_k, t \rangle = \langle T, W^k, (P^{k-1}(x),t) \rangle,
\end{equation}
\item for $\lebmeas^{n-k}$-a.e. $b \in R^{n-k}$,
\begin{equation}
\langle \langle T, W^{k-1}, P^{k-1}(x) \rangle, \underline{W}^k , b \rangle = \langle T, \hat{W}_k, (P^{k-1}(x), b) \rangle,
\end{equation}
\item the following two iterated-slicing equalities hold
\begin{align}
\langle \langle T, W^{k-1}, P^{k-1}(x) \rangle, W_k, P_k(x) \rangle 
&= \langle T, W^k, P^k(x) \rangle,\\
\langle \langle T, W^{k-1}, P^{k-1}(x)\rangle, \underline{W}^k, \underline{P}^k(x) \rangle 
&= \langle T, \hat{W}^k, \hat{P}_k(x) \rangle.
\end{align}
\end{itemize}
We next apply Lemma \ref{le:GoodCuts} to (a scaled version of) $K^r$. 
The conclusion of that lemma is that there exist compact sets $A^k \subset [0,1]^{k}$, ($k=1, \dots, n$), with $A^n \subset K^r$ such that for $k=1, \dots, n$, $\Ha^k(A^k) > (1 - \delta) (2r)^k$, $(A^k \times \R^{n-k} )\cap A^n = A^n$ and for all $x \in A^k$, it holds that
\begin{equation}
\Ha^{n-k} ( (\{ x \} \times \R^{n-k}) \cap A^n  ) > (1 - \delta)(2r)^{n-k}.
\end{equation}

We will set $A^0 = \{0\}$, and we will use the convention that $\{0 \} \times \R^{n} = \R^{n}$, and $T =\langle T, W^0, 0 \rangle$.

To keep notation concise, we denote for $a \in A^{k}$,
\begin{equation}
K(a) := ( \{ a\} \times \R^{n-k} ) \cap A^n.
\end{equation}

We will show that for $k = 0, \dots, n-1$, for $a \in A^{k}$, for all $f$ Lipschitz on $\supp T$, with $\tilde{f} = g_1^\# f = g_1 \circ f$, it holds that
\begin{equation}
\label{eq:PoincareInduction}
\int_{K(a)} | \tilde{f} - (\tilde{f})_{K(a)} |d\Ha^{n-k} 
\leq C(n,k) r \int_{B_R(g_1(x))} | df | d\| \langle T, W^k, a \rangle \|.
\end{equation}
To be precise, by $|df|$ we mean the dual norm of the tangential derivative of $f$, as $f$ is restricted to $\set (\langle T, W^k, a \rangle )$.
We will show that (\ref{eq:PoincareInduction}) holds by induction on $k$, starting at $k = n-1$. 

Let $k=n-1$. Let $a \in A^{n-1}$, and let $t_1 < t_2$ be such that $(a,t_i) \in A^n$ for $i=1,2$. 
Then, since $A^n \subset K^r$, it follows that there is a $\sigma:[0,L] \to Y$ with $\sigma(0) = g_1((a,t_1))$, $\sigma(L) = g_1((a,t_2))$, $\sigma([0,L]) \subset B_R(g_1(x))$, and $\sigma([0,L])$ concentrated on a subset of $\cup_\ell g_\ell(K_\ell)$. 
In particular, $f$ is defined on the image of $\sigma$.
Consequently,
\begin{equation}
\label{eq:DetailedkOne}
\begin{split}
|\tilde{f}(a,t_2) - \tilde{f}(a,t_1)| 
& = | f(g_1((a,t_2))) -  f(g_1((a,t_1))) | \\
& = | f(\sigma(L)) - f(\sigma(0)) | \\
& \leq \int_0^L |(f \circ \sigma)' (s)| ds \\
& \leq \int_0^L | df | |\sigma'(s)| ds \\
& \leq \int_{B_R(g_1(x))} |df| d \|T, \hat{W}_n, a \|.
\end{split}
\end{equation}
By construction, $K(a) \geq (1 - \delta) 2r$.
We apply Jensen's inequality to conclude
\begin{equation}
\begin{split}
|\tilde{f}(a,t_1) - (\tilde{f})_{K(a)}|
&\leq \frac{1}{\Ha^1(K(a))} \int_{t_2: (a,t_2) \in K(a)} |\tilde{f}(a,t_1) - \tilde{f}(a,t_2)| dt_2 \\
&\leq \frac{2r}{(1 - \delta) 2r} \int_{B_R(g_1(x))} |df| d \| \langle T, \hat{W}_n, a \rangle \|.
\end{split}
\end{equation}
It follows that
\begin{equation}
\int_{K(a)} |\tilde{f} - (\tilde{f})_{K(a)}| d \Ha^{1} \leq \frac{2r}{1 - \delta} \int_{B_R(g_1(x))} |df| d \|\langle T, \hat{W}_n, a\rangle \|,
\end{equation}
or in other words, that (\ref{eq:PoincareInduction}) for $k=n-1$ with $C(n,n-1) = 4$.

Let the inequality (\ref{eq:PoincareInduction}) be shown for $k$.
We will show the inequality holds for $k-1$.

Let $a \in A^{k-1}$. 
We first estimate as follows
\begin{equation}
\label{eq:PoinStartSplit}
\begin{split}
\int_{K(a)}& |\tilde{f} - (\tilde{f})_{K(a)}| d\Ha^{n-k+1} \\
& \leq \int_{t:(a,t) \in A^{k}} \int_{K(a,t)} |\tilde{f} - (\tilde{f})_{K(a,t)}| d\Ha^{n-k} d\Ha^1(t) \\
& \quad + \int_{t:(a,t) \in A^{k}} \int_{K(a,t)} |(\tilde{f})_{K(a,t)} - (\tilde{f})_{K(a)}| d\Ha^{n-k} d\Ha^1(t).
\end{split}
\end{equation}
We can bound the first term by using the induction hypothesis
\begin{equation}
\label{eq:PoinEasyTerm}
\begin{split}
& \int_{t:(a,t) \in A^{k}} \int_{K(a,t)} |\tilde{f} - (\tilde{f})_{K(a,t)}| d\Ha^{n-k} d\Ha^1(t)\\
& \leq C(n,k) r\int_{t:(a,t) \in A^{k}} \int_{B_R(g_1(x))} |d f| d \|\langle T, W^{k}, (a, t) \rangle \| dt \\ 
& = C(n,k) r\int_{t:(a,t) \in A^{k}} \int_{B_R(g_1(x))} |d f| d \|\langle \langle T, W^{k-1}, a \rangle, W_k, t \rangle \| dt \\ 
& \leq 2 \sqrt{n} \, C(n,k) r \int_{B_R(g_1(x))} |d f| d \|\langle T, W^{k-1}, a \rangle \|.
\end{split}
\end{equation}
(Technically, the meaning of $|df|$ changed (and in general became larger $\|\langle T, W^{k-1}, a\rangle\|$-a.e.) in the last step of the previous estimate.)

Let $t_1, t_2 \in [x_{n-k} - r, x_{n-k} + r]$, such that
\[
a^i := (a,t_i) = (a_1, \dots, a_{n-k-1}, t_i) \in A^{k}, \qquad i=1,2.
\]
By the induction hypothesis,
\begin{equation}
\label{eq:PoincareByIndHyp}
\int_{K(a^i)} | \tilde{f}- (\tilde{f})_{K((a,t_i))} | d\Ha^{n-k} 
\leq C(n,k) r \int_{B_R(g_1(x))} |d f| d\| \langle T, W^{k}, (a,t_i) \rangle \|.
\end{equation}
Since $(a,t_i) \in A^{k}$, there exists a compact set $B \subset \R^{n-k}$, with $\Ha^{n-k}(B) \geq (1 - 2 \delta)(2r)^{n-k}$ such that for $(b_1, \dots, b_{n-k}) \in B$,
\begin{equation}
(a_1, \dots, a_{k-1},t_i,b_1, \dots, b_{n-k}) \in A^n, \qquad i = 1,2.
\end{equation}
We apply Jensen's inequality and use (\ref{eq:PoincareByIndHyp}) to estimate
\begin{equation}
\label{eq:AvRestrClose}
\begin{split}
| (\tilde{f})_{\{(a,t_i)\} \times B} - (\tilde{f})_{K((a,t_i))} | 
& \leq \frac{1}{\Ha^{n-k}(B)} \int_{\{(a,t_i)\} \times B} | \tilde{f} - (\tilde{f})_{K((a,t_i))}| d \Ha^{n-k} \\
& \leq \frac{C(n,k) r}{(1-2\delta)(2 r)^{n-k}} \int_{B_R(g_1(x))} |df| d\|\langle T,  W^{k}, (a,t_i) \rangle \|.
\end{split}
\end{equation}
Moreover, let 
\begin{equation}
(a^i, b) := (a_1, \dots, a_{n-k-1}, t_i, b_1, \dots, b_k) \in A^n,
\end{equation}
then there exists a component $\sigma:[0,L] \to Y$ of 
\begin{equation}
\langle T, \hat{W}_{k}, (a,b) \rangle  \llcorner \{ t_1 < W_{k} \leq t_2 \},
\end{equation}
such that $\Lip(\sigma) \leq 1$, 
\begin{equation}
\begin{cases}
\sigma(0) = (a^1,b), \quad \sigma(L) = (a^2,b), & \text{if } n - k \text{ even},\\
\sigma(0) = (a^2, b), \quad \sigma(L) = (a^1, b), &\text{if } n - k \text{ odd},
\end{cases}
\end{equation}
and $\sigma([0,L]) \subset B_R(g_1(x))$, and therefore
\begin{equation}
\begin{split}
|\tilde{f}((a^2,b)) - \tilde{f}((a^1,b))| 
& \leq \int_{B_R(g_1(x))}  |df| d\|\sigma_\# \llbracket \chi_{[0,L]} \rrbracket \| \\
& \leq \int_{B_R(g_1(x))}  |df| d\| \langle T, \hat{W}_{k}, (a,b) \rangle \|.
\end{split}
\end{equation}
We integrate this inequality to obtain
\begin{equation}
\begin{split}
|(f)_{\{a^1\} \times B} - (f)_{\{a^2\} \times B}| 
& \leq \frac{ 1 }{\Ha^k (B)} \int_{B} \int_{B_R(g_1(x))} |df| d\| \langle T, \hat{W}_{k}, (a,b) \rangle  \| d\Ha^{n-k} \\
& \leq \frac{ (2\sqrt{n})^{n-k} }{(1- 2 \delta) (2r)^{n-k}} \int_{B_R(g_1(x))} |d f | d \| \langle T, W^{k-1}, a \rangle\|.
\end{split}
\end{equation}
We combine this with (\ref{eq:AvRestrClose}) and use the triangle inequality to obtain
\begin{equation}
\begin{split}
|(\tilde{f})_{K(a^1)} - (\tilde{f})_{K(a^2)}| 
&\leq \sum_{i=1,2} \frac{C(n,k) r}{(1-2 \delta)(2r)^{n-k}}\int_{B_R(g_1(x))} |df|  d\|\langle T, W^{k}, (a,t_i) \rangle\| \\
& \quad + \frac{(2 \sqrt{n})^{n-k}}{(1-2\delta)(2r)^{n-k}} \int_{B_R(g_1(x))} |df| d\|\langle T, W^{k-1}, a \rangle\|.
\end{split}
\end{equation}
It follows, again by Jensen's inequality, that 
\begin{equation}
\begin{split}
&|(\tilde{f})_{K((a,t_1))} - (\tilde{f})_{K(a)}| \\
&= \frac{1}{\Ha^{n-k+1}( K(a) ) } \left| \int_{t_2} \left( (\tilde{f})_{K((a,t_1))} - (\tilde{f})_{K((a,t_2))} \right)\Ha^{n-k}(K((a,t_2)) ) dt_2 \right| \\
&\leq \frac{1}{(1 - \delta) (2 r) } \int_{t_2 : (a,t_2) \in A^{n-k} } \left| (\tilde{f})_{K((a,t_1))} - (\tilde{f})_{K((a,t_2))} \right| dt_2.
\end{split}
\end{equation}
Since $\Ha^k(K(a,t_2)) \neq 0$ if and only if $(a,t_2) \in A^{n-k}$,
\begin{equation}
\begin{split}
& \int_{t_1 : (a,t_1) \in A^{n-k} } \int_{K(a,t_1)} |(\tilde{f})_{K((a,t_1))} - (\tilde{f})_{K(a)}| d \Ha^{n-k} dt_1 \\
& \leq (2r)^{n-k} \int_{t_1 : (a,t_1) \in A^{n-k} } |(\tilde{f})_{K((a,t_1))} - (\tilde{f})_{K(a)}| dt_1 \\
& \leq \frac{2r (2 \sqrt{n})^{n-k}}{(1- 2 \delta)(1- \delta)} \int_{B_R(g_1(x))} |df| d \| \langle T, W^{k-1}, a \rangle \| \\
& \quad + \frac{2 r C(n,k) 2 \sqrt{n}}{ (1 - 2\delta) (1-\delta) } \int_{B_R(g_1(x))} |df| d \| \langle T, W^{k-1}, a \rangle \|.
\end{split}
\end{equation}
This concludes the estimate of the second term on the right-hand side of (\ref{eq:PoinStartSplit}). 
Together with (\ref{eq:PoinEasyTerm}), we conclude that indeed
\begin{equation}
\int_{K(a)} | \tilde{f} - (\tilde{f})_{K(a)} |d\Ha^{n-k+1} 
\leq C(n,k-1) r \int_{B_R(p)} | df | d\| \langle T, W^{k-1}, a \rangle \|,
\end{equation}
with $C(n,k-1) = 4 (2\sqrt{n}) C(n,k) + 4 (2 \sqrt{n})^{n-k}$.
\end{proof}

\begin{lemma}
\label{le:GoodCuts}
Let $K \subset [0,1]^n$ be a compact set such that $|K| > 1 - \epsilon$. 
Let $\delta > 0$.
Then there exist compact sets $A^k \subset [0,1]^{k}$, for $k=1, \dots, n$, such that $A^n \subset K$, 
and for all $k = 1, \dots, n$, $\Ha^{k}(A^{k}) > 1 - \epsilon / \delta^n$, and
\begin{equation}
(A^{k} \times \R^{n-k}) \cap A^n = A^n,
\end{equation}
and for all $x \in A^{k}$,
\begin{equation}
\Ha^{n-k}  ((\{ x \} \times \R^{n-k} ) \cap A^n ) > 1 - \delta.
\end{equation}
\end{lemma}

\begin{proof}
First, we will inductively define compact sets $A_{k} \subset [0,1]^{k}$ (not necessarily equal to the $A^k$), for $k = 1,\dots,n$, starting with $A_n := K$. 
We also set $K_n := K$, 
\begin{equation}
\begin{split}
K_{k} &:= K_{k+1} \cap (A_k \times \R^{n-k}) \\
&= K \cap (A_{n-1} \times \R^1) \cap (A_{n-2} \times \R^2) \cap \dots \cap (A_k \times \R^{n-k}),
\end{split}
\end{equation}
At the same time, we will show that $\Ha^{k}(A_{k}) > 1 - \epsilon / \delta^{n-k}$ and $\Ha^{n}(K_{k}) > 1 - \epsilon /  \delta^{n - k}$.

Let $A_n = [0,1]^n$ and $K_n = K$.
With $A_{k+1}$ defined, we proceed to define $A_k$ as follows.
Let 
\begin{equation}
\tilde{B}_{k} := \{ x \in [0,1]^{k} \, | \, \Ha^{n-k}( (\{x\} \times \R^{n-k}) \cap K_{k+1} ) \leq 1 - \delta \}.
\end{equation}
Since $\Ha^{n}( K_{k+1} ) > 1 - \epsilon / \delta^{n-k-1}$, by Fubini's Theorem,
\begin{equation}
(1 - \delta)\Ha^{k}(\tilde{B}_{k})+ (1 - \Ha^{k}(\tilde{B}_{k}))  > 1 - \frac{\epsilon}{\delta^{n-k-1}},
\end{equation}
such that $\Ha^{k}(\tilde{B}_{k}) < \epsilon / \delta^{n-k}$.
Select a relatively open $B_{k} \subset [0,1]^{k}$ such that $\tilde{B}_{k} \subset B_{k}$ and 
\begin{equation}
\Ha^{k}(B_{k}) < \frac{\epsilon}{\delta^{n-k}}, \qquad \Ha^n(K_{k+1}) -  \Ha^{k}(B_{k}\backslash \tilde{B}_{k}) > 1 - \frac{\epsilon}{ \delta^{n-k-1} }.
\end{equation}
Put $A_{k} = [0,1]^{k} \backslash B_{k}$. 
Note that $A_{k}$ is compact. 
Set $K_{k} :=  K_{k+1} \cap (A_{k} \times \R^{n-k})$. 
Note that 
\begin{equation}
\begin{split}
\Ha^{n}(K_{k}) 
&> \Ha^{n}(K_{k+1}) - (1 - \delta) \Ha^{k}(\tilde{B}_{k}) -  \Ha^{k}(B_{k}\backslash \tilde{B}_{k}) \\
&> 1 - \frac{\epsilon}{\delta^{n-k-1}} - (1 - \delta)\frac{\epsilon}{\delta^{n-k}}\\
&= 1 - \frac{\epsilon}{\delta^{n-k}}.
\end{split}
\end{equation}
Now define $A^k \subset [0,1]^k$ by $A^n := K_1$ and 
\begin{equation}
A^{k} := (A_1 \times \R^{k-1}) \cap (A_2 \times \R^{k-2} ) \cap  \dots \cap A_{k},
\end{equation}
for $k = 1, \dots, n-1$. We find that by the definition of $A^n$,
\begin{equation}
(A^{k} \times \R^{n-k}) \cap A^n = A^n,
\end{equation}
so that by Fubini
\begin{equation}
\Ha^{k}(A^{k}) \geq \Ha^n (K_1)  > 1 - \epsilon/ \delta^n.
\end{equation}
If $x \in A^{k}$, then 
\begin{equation}
\begin{split}
(\{ x \} \times \R^{n-k}) \cap A^n &= (\{ x \} \times \R^{n-k}) \cap (A^k \times \R^{n-k}) \cap K_{k} \\
& = (\{ x \} \times \R^{n-k}) \cap K_k \\
& > 1 - \delta.
\end{split}
\end{equation}

\end{proof}

\bibliography{IFEVrefCVPDE}
\bibliographystyle{spmpsci} 

\end{document}